\documentclass[reqno, 12pt]{amsart}

\usepackage[utf8]{inputenc}
\usepackage[T1]{fontenc}
\usepackage[english]{babel}
\usepackage{libertine} 
\usepackage[libertine]{newtxmath}
\usepackage{fullpage}
\usepackage{xcolor}
\usepackage{indentfirst,enumitem}
\usepackage[linkcolor=blue, citecolor=green, colorlinks=true]{hyperref}

\allowdisplaybreaks

\title{Conditional lower bounds on the distribution of central values: the case of modular forms}
\author{}
\date{\today}

\author[Lesesvre]{Didier Lesesvre}
\address{Universit\'e de Lille -- Laboratoire Paul Painlev\'e, UMR 8524,
 59000 Lille, France}
\email{didier.lesesvre@univ-lille.fr}

\author[Suriajaya]{Ade Irma Suriajaya}
\address{Faculty of Mathematics, Kyushu University, 744 Motooka, Nishi-ku, Fukuoka 819-0395, Japan}
\email{adeirmasuriajaya@math.kyushu-u.ac.jp}

\newtheorem{prop}{Proposition}
\newtheorem{thm}{Theorem}
\newtheorem{lem}{Lemma}
\newtheorem{coro}{Corollary}
\newtheorem{conj}{Conjecture}

\theoremstyle{remark}
\newtheorem{rmk}{Remark}

\newcounter{daggerfootnote}
\newcommand*{\daggerfootnote}[1]{%
    \setcounter{daggerfootnote}{\value{footnote}}%
    \renewcommand*{\thefootnote}{\fnsymbol{footnote}}%
    \footnote[2]{#1}%
    \setcounter{footnote}{\value{daggerfootnote}}%
    \renewcommand*{\thefootnote}{\arabic{footnote}}%
    }

\renewcommand{\Re}{\mathrm{Re}}
\renewcommand{\Im}{\mathrm{Im}}
\renewcommand{\ge}{\geqslant}
\renewcommand{\le}{\leqslant}

\numberwithin{equation}{section}

\newcommand{\sumh}{\sideset{}{^h}\sum}

\begin{document}

\begin{abstract}
Radziwi{\l}{\l} and Soundararajan unveiled a connection between low-lying zeros and central values of $L$-functions, which they instantiated in the case of quadratic twists of an elliptic curve. This paper addresses the case of the family of modular forms in the level aspect, and proves that the logarithms of central values of associated $L$-functions approximately distribute along a normal law with mean $-\tfrac12 \log \log c(f)$ and variance $\log\log{c(f)}$, where $c(f)$ is the analytic conductor of~$f$, as predicted by the Keating-Snaith conjecture.
\end{abstract}

\maketitle

\section{Introduction}
\label{sec:intro}

In practice, $L$-functions appear as generating functions encapsulating information about various objects, such as Galois representations, elliptic curves, arithmetic functions, modular forms, Maass forms, etc. 
Studying $L$-functions is therefore of utmost importance in number theory at large. Two of their attached data carry critical information: their zeros, which govern the distributional behavior of underlying objects; and their central values, which are related to invariants such as the class number of a field extension. We refer to \cite{is} and references therein for further hindsight.

\subsection{Distribution of zeros}

The spacings of zeros of families of $L$-functions are well-understood: they are distributed according to a universal law, independent of the exact family under consideration, as proven by Rudnick and Sarnak \cite{rs95}. This recovers the behavior of spacings between eigenangles of the classical groups of random matrices.
However the distribution of \textit{low-lying} zeros, i.e. those located near the real axis, attached to reasonable families of $L$-functions does depend upon the specific setting under consideration. See~\cite{sst} for a discussion in a general setting.

More precisely, let $L(s, f)$ be an $L$-function attached to an arithmetic object $f$. Consider its nontrivial zeros written in the form $\rho_f = \frac{1}{2}+i\gamma_f$ where $\gamma_f$ is a priori a complex number. There is a notion of analytic conductor $c(f)$ of $f$ quantifying the number of zeros of $L(s,f)$ in a given region, more precisely such that the number of zeros with rela and imaginary parts between $0$ and $1$ is approximately $\log (c(f)) / 2\pi$;
we renormalize the mean spacing of the zeros to $1$ by setting $\tilde{\gamma}_f = {\log (c(f))} \gamma_f / {2\pi}$.
Let $h$ be an even Schwartz function on $\mathbb{R}$ whose Fourier transform is compactly supported, in particular it admits an analytic continuation to all $\mathbb{C}$. The one-level density attached to $f$ is defined by
\begin{equation}
\label{def:one-level-density}
D(f, h) := \sum_{\gamma_f} h\left(\tilde{\gamma}_f\right).
\end{equation}
The analogy with the behavior of small eigenangles of random matrices led Katz and Sarnak to formulate the so-called \emph{density conjecture}, claiming the same universality for the types of symmetry of families\daggerfootnote{The notion of suitable \textit{family} of automorphic forms may depend upon the setting and the required properties. Sarnak \cite{sarnak} provides a wide list of examples of families arising in practice and of operations allowing to derive new families from those. See also \cite{sst} for further discussions.} of $L$-functions as those arising for classical groups of random matrices.

\begin{conj}[Katz-Sarnak] 
\label{conj:katz-sarnak}
Let $\mathcal{F}$ be a family$^\dagger$ of $L$-functions and $\mathcal{F}_X$ a finite truncation increasing to $\mathcal{F}$ when $X$ grows. Then there is one classical group $G$ among $\mathrm{U}$, $\mathrm{SO(even)}$, $\mathrm{SO(odd)}$, $\mathrm{O}$ or $\mathrm{Sp}$ such that for all even Schwartz function $h(x)$ on $\mathbb{R}$ with compactly supported Fourier transform, 
\begin{equation}
\label{density-conjecture}
\frac{1}{|\mathcal{F}_X|} \sum_{f \in \mathcal{F}_X} D(f, h) \xrightarrow[X\to\infty]{} \int_{\mathbb{R}} h(x) W_G(x)dx, 
\end{equation}
where $W_G(x)$ is the explicit distribution function modeling the distribution of the eigenangles of the corresponding group of random matrices, explicitly $W_{\rm U}(x) = 1$ and
\begin{equation}
    \begin{array}{rclcrcl}
         W_{\mathrm{O}}(x) & =  &\displaystyle 1 + \frac{1}{2}\delta_0(x), & \qquad &
W_{\mathrm{SO(even)}}(x) & =  &\displaystyle 1+ \frac{\sin 2\pi x}{2 \pi x}, \\[1em]
W_{\mathrm{SO(odd)}}(x) & = &\displaystyle  1 - \frac{\sin 2\pi x}{2\pi x} + \delta_0(x), & \qquad &
W_{\mathrm{Sp}}(x) & = &\displaystyle   1 - \frac{\sin 2\pi x}{2\pi x}.
    \end{array}
\end{equation}

The family $\mathcal{F}$ is then said to have the \emph{type of symmetry} of~$G$.
\end{conj}

Various results towards this conjecture have been established in the recent two decades, following the seminal paper of Iwaniec, Luo and Sarnak \cite{ils}; we refer to \cite{sst} for a general discussion and various references.

\subsection{Distribution of central values}

The distribution of central values of $L$-functions is also finely understood, and the Keating-Snaith conjecture predicts that the logarithmic central values $\log L(\tfrac12, f)$ are asymptotically distributed according to a normal distribution, with explicit mean and variance depending on the family.
 
\begin{conj}[Keating-Snaith] 
\label{conj:keating-snaith}
Let $\mathcal{F}$ be a reasonable family of $L$-functions in the sense of Sarnak, and $\mathcal{F}_X$ a finite truncation increasing to $\mathcal{F}$ when $X$ grows. 
There is a mean $M_{\mathcal{F}}$ and a variance $V_{\mathcal{F}}$ such that for any real numbers $\alpha < \beta$, 
\begin{equation}
\label{keating-snaith-conjecture}
\frac{1}{|\mathcal{F}_X|}\left| \left\{ f \in \mathcal{F}_X \ : \ \frac{\log L(\tfrac12, f) - M_{\mathcal{F}}}{V_{\mathcal{F}}^{1/2}} \in (\alpha, \beta) \right\} \right| \xrightarrow[X\to\infty]{} \frac{1}{\sqrt{2\pi}} \int_\alpha^\beta e^{-x^2/2} dx.
\end{equation}
In particular, the family of the logarithmic central values $\log L(\tfrac12, f)$ equidistributes asymptotically with respect to a normal distribution.
\end{conj}

\begin{rmk}
    In the above statement, in the case $L(\tfrac12, f)=0$ its logarithm is understood as $-\infty$. Conjecturally, the central value is always non-negative, as it can be seen assuming the generalized Riemann hypothesis and using a continuity argument on the real line. In certain cases, positivity can be obtained directly by applying Waldspurger's formula, but it remains unknown in general. 
\end{rmk}

\subsection{Relation between both conjectures}

Radziwi\l\l{} and Soundararajan \cite{rs} claimed a general principle that any restricted result towards Conjecture \ref{conj:katz-sarnak} can be refined to show that most such $L$-values have the typical distribution predicted by Conjecture~\ref{conj:keating-snaith}. They instantiated this technique in the case of quadratic twists of a given elliptic curve and suggested the wide applicability of this approach, in particular in the case of modular forms building on the pioneering work of Iwaniec, Luo and Sarnak \cite{ils}. This paper shows that this principle indeed holds and provides the proof in the case of modular forms in the level aspect.

More precisely, for integers $k \geqslant 2$ and $q \geqslant 1$, let $H_k(q)$ be an orthogonal basis of Hecke eigenforms, which is a basis of the space of newforms $S_k^{\rm new}(q)$, normalized so that their first Fourier coefficients are $a_f(1) = 1$. We let $c(f) = k^2q$ be the analytic conductor \cite{is} of $f$. Introduce for a general sequence $(a_f)_{f \in H_k(q)}$ the harmonic average
\begin{equation}
\label{harmonic-weights}
\sumh_{f \in H_k(q)} a_f := \frac{\Gamma(k-1)}{(4\pi)^{k-1}} \sum_{f \in H_k(q)} \frac{a_f}{\|f\|^2}
\end{equation}
which includes the suitable weights in order to apply the Petersson trace formula. In this setting, the seminal work of Iwaniec, Luo and Sarnak \cite{ils} as well as the recent achievement of Baluyot, Chandee and Li \cite{bcl} obtain the following restricted statement towards Conjecture \ref{conj:katz-sarnak}.
\begin{thm}[Iwaniec, Luo, Sarnak \& Baluyot, Chandee, Li]
For any smooth function $\Psi$ compactly supported and any Schwartz function $h$ such that its Fourier transform $\widehat{h}$ is supported in $(-4, 4)$, we have 
\begin{equation}
\label{thm:ilsbcl}
\frac{1}{N(Q)} \sum_{q \geqslant 1} \Psi\left( \frac{q}{Q}\right) \sumh_{f \in H_k(q)} D(f, h) \xrightarrow[Q \to \infty]{} \int_{\mathbb{R}} W_{\rm O} h = \widehat{h}(0) + \frac{1}{2}h(0), 
\end{equation}
where $W_{\rm O} = 1 + \tfrac12\delta_0$ is the orthogonal density and $N(Q)$ is the weighted cardinality of the family,
\begin{equation}
\label{smooth-cardinality}
N(Q) := \sum_{q \geqslant 1} \Psi\left( \frac{q}{Q}\right) \sumh_{f \in H_k(q)} 1.
\end{equation}
\end{thm}

Building on this result and exploiting the methodology outlined by Radziwi\l\l{} and Soundararajan, we prove the following statement towards Conjecture \ref{conj:keating-snaith}.
\begin{thm}
\label{thm:thm}
For any $q \geqslant 1$ and $k \geqslant 2$, let $H_k(q)$ be an orthogonal basis of Hecke eigenforms of level $q$, weight $k$ and trivial nebentypus, which is also a basis of the space of newforms~$S_k^{\rm new}(q)$, normalized so that their first Fourier coefficients are $a_f(1) = 1$. Assume the generalized Riemann hypothesis for the associated symmetric squares $L$-functions $L(s, \mathrm{sym}^2 f)$.

For any smooth function $\Psi$ compactly supported and for any real numbers $\alpha < \beta$, we have
\begin{equation*}
\label{thm-eq}
\frac{1}{N(Q)} \sum_{q \geqslant 1} \Psi\left( \frac{q}{Q}\right) \left| \left\{ f \in H_k(q) \ : \ \tfrac{\log L\left(\tfrac12,f\right) + \tfrac12\log\log c(f)}{\sqrt{\log \log c(f)}} \in (\alpha, \beta) \right\} \right| \geqslant \frac{5}{8} \frac{1}{\sqrt{2\pi}} \int_\alpha^\beta e^{-x^2/2} dx + o(1).
\end{equation*}
\end{thm}

This result is in phase with Conjecture \ref{conj:keating-snaith} with $M_{\mathcal{F}} = -\tfrac12 \log \log c(f)$ and $V_{\mathcal{F}} = {\log \log c(f)}$.

\begin{rmk}
    In the case of a modular form with sign $\varepsilon_f = -1$ in the functional equation, its associated central $L$-value vanish, so that in such a case there is no hope to obtain a lower bound with a constant $1$ towards the Keating-Snaith conjecture (which hides such issues in the notion of "reasonable" family). 
    Using the sieving technique from \cite{ils}, it is possible to isolate the modular forms having positive sign in the functional equation, and the same approach as the one presented here would yield a constant $13/16$ towards Conjecture \ref{conj:katz-sarnak}.
\end{rmk}

\subsection{Strategy of proof and structure of the paper} 

In Section \ref{sec:oddsandends} we recall the needed definitions on modular forms and $L$-functions. In particular, explicit formulas relate central values of $L$-functions to sums of modular coefficients over primes,
\begin{equation}
\log L(\tfrac12, f) \leadsto \sum_{p \leqslant x} \frac{a_f(p)}{\sqrt{p}} - \frac{1}{2}\log \log x,
\end{equation}
so that the claimed mean is already displayed and most of the study reduces to understanding the distribution of the above sums over primes, denoted $P(f,x)$, as well as the error in the explicit formula, which can be expressed as a sum over zeros of $L$-functions. In order to study the distribution of the sums over primes~$P(f,x)$, we appeal to the moment method and examine the behavior of their powers~$P(f,x)^k$. Regrouping equal primes together, this leads to consider sums of the form
\begin{equation}
\label{ff}
    \sum_{\substack{p_1, \ldots, p_\ell \\ p_i \neq p_j }} \frac{a_f(p_1)^{\alpha_1} \cdots a_f(p_\ell)^{\alpha_\ell}}{p_1^{\alpha_1} \cdots p_\ell^{\alpha_\ell}},
\end{equation}
where the $p_i$ are prime numbers and the $\alpha_i$ are positive integers. The study of these sums constitutes the heart of the paper and requires to finely understand sums of Fourier coefficients of modular forms.
In Section \ref{sec:moments}, the contribution of the terms with $\alpha_i = 2$ is shown to be the dominant term, and to match the corresponding moment of the normal distribution i.e. the main term in Theorem~\ref{prop:moments}. We inductively reduce the study of the other sums of the form \eqref{ff} to the case where the only powers arising are~$\alpha_i = 1$; these are then inductively shown to contribute as an error term by using the harmonic average and trace formulas, a strategy implemented in Section~\ref{sec:only-ones}.
Section~\ref{sec:conclusion} concludes the proof by showing that the extra terms arising in the explicit formula, in the guise of sum over zeros, are negligible except for a small proportion of modular forms. This is the place where limited results towards the distribution of low-lying zeros are used, and is the origin of the loss in Theorem \ref{thm:thm} compared to the Keating-Snaith conjecture.

\begin{rmk}
Radziwi\l\l{} and Soundararajan \cite{rs} outline a general strategy to prove such results. However, in their specific case of quadratic twists of elliptic curves, they rely on the Poisson summation formula to estimate character sums, as well as the complete multiplicativity of characters. These tools are however not available in the case of modular coefficients, and it requires the inductive use of Hecke relations instead of multiplicativity, and of trace formulas instead of Poisson summation formula.
\end{rmk}

\section{Odds and ends}
\label{sec:oddsandends}

\subsection{Modular \textit{L}-functions}

We start recalling the needed theory of modular forms, referring to~\cite{iwaniec_topics} for a detailed account.
Let $S_k(q)$ be the space of holomorphic cusp forms of weight $k$, level~$q$ and trivial nebentypus. A cusp form $f \in S_k(q)$ has an attached $L$-function defined by
\begin{equation}
\label{dirichlet-series}
L(s,f) = \sum_{n = 1}^\infty \frac{a_f(n)}{n^s},
\end{equation}
where the $a_f(n)$ are its Fourier coefficients, defined by the Fourier expansion
\begin{equation}
\label{fourier-expansion}
f(z) = \sum_{n \geqslant 1} a_f(n) n^{(k-1)/2} e(nz).
\end{equation}
The modular forms are arithmetically normalized, i.e. we assume that $a_f(1) = 1$. In this normalization, Deligne's bound states that $a_f(n) \ll d(n) \ll n^\varepsilon$, where $d(n)$ denotes the divisor function. In particular, the Dirichlet series \eqref{dirichlet-series} converges for all $\Re(s) > 1$. The degree two $L$-function~$L(s,f)$ can be completed by explicit gamma factors \cite[Section 5.11]{ik} so that we have the functional equation
\begin{equation}
\label{completed-L-function}
\Lambda(s,f) := (2\pi)^{-s} q^{s/2} \Gamma(s+\tfrac{k-1}{2})\Gamma(s+\tfrac{k+1}{2})L(s,f) = \varepsilon_f \Lambda(1-s,f),
\end{equation}
where $\varepsilon_f \in \{\pm 1\}$ is the root number of $f$. 

If $f \in S_k(q)$ is an eigenfunction of all the Hecke operators $T_n$ for $(n,q)=1$, we say that $f$ is a Hecke eigenform. If it moreover lies in the orthogonal complement of the space of the oldforms, i.e. those of the form $f(z) = g(dz)$ for a certain $g \in S_k(q/d)$ where $d \mid q$, then we say that~$f$ is a newform, case in which it is an eigenform for the Hecke operators $T_n$ for \textit{all} $n \geqslant 1$. Let $H_k(q) \subset S_k(q)$ be an orthogonal basis of the space of newforms consisting of Hecke eigenforms~$f$.
For $f \in H_k(q)$, we have the Euler product
\begin{equation}
\label{euler-product}
L(s,f) = \prod_p (1-a_f(p)p^{-s} + p^{-2s})^{-1} = \prod_p (1-\alpha_f(p) p^{-s})^{-1} (1-\beta_f(p)p^{-s})^{-1},
\end{equation}
where the product is over prime numbers $p$, and $\alpha_f(p), \beta_f(p) \in \mathbb{C}$ are called the spectral parameters of $f$ at $p$. This expression encapsulates the Hecke relations satisfied by the coefficients. By taking the logarithmic derivative of this expression, we obtain
\begin{equation}
\label{log-derivative}
-\frac{L'}{L}(s,f) = \sum_{n \geqslant 1} \frac{\Lambda_f(n)}{n^s},
\end{equation}
where $\Lambda_f(n) = (\alpha_f(p)^k + \beta_f(p)^k) \log(p)$ if $n = p^k$ is a prime power, and $\Lambda_f(n) = 0$ otherwise. 

We assume the generalized Riemann hypothesis for the symmetric squares $L$-functions $L(s, \mathrm{sym}^2 f)$ all along the paper.

\subsection{Explicit formula for sums over zeros}

We have the celebrated Weil explicit formula, proven for instance in \cite[(4.11)]{ils}, relating sum over zeros of $L$-functions and sum over primes of their spectral parameters. For any smooth function $h$ with compact Fourier support, we have
\begin{equation}
\label{explicit-formula-zeros}
D(f,h) = \widehat{h}(0) - \frac{2}{\log c(f)} \sum_{p} \sum_{\nu \geqslant 1} (\alpha_f(p)^\nu + \beta_f(p)^\nu) \frac{\log p}{p^{\nu/2}} \widehat{h}\left( \frac{\nu \log p}{\log c(f)}\right) + O\left( \frac{1}{\log c(f)}\right).
\end{equation}

Using the relations between coefficients and spectral parameters when $p \nmid q$, and the Deligne bounds on~$a_f(n)$, we obtain that the terms $\nu \geqslant 3$ contribute no more than the displayed error term, so that we deduce as in~\cite[Lemma 4.1]{ils} the following expansion of the one-level density.
\begin{prop}[Explicit formula for sums over zeros]
\label{prop:explicit-formula-zeros}
We have, for all $f \in H_k(q)$ and all smooth function $h$, 
\begin{equation}
\label{d-expansion-ils}
D(f, h) = \widehat{h}(0) + \tfrac12h(0) + P^{(1)}(f, h) + P^{(2)}(f, h) + O\left(\frac{\log\log c(f)}{\log c(f)}\right),
\end{equation}
where, for $\nu \geqslant 1$, we let 
\begin{equation}
\label{Pnu}
P^{(\nu)}(f,h) = \frac{2}{\log c(f)} \sum_{p\nmid q} a_f(p^\nu) \frac{\log p}{p^{\nu/2}} \widehat{h}\left( \frac{\nu \log p}{\log c(f)}\right), 
\end{equation}
where the sum runs over prime numbers $p$ not dividing $q$.
\end{prop}

\begin{rmk}
    The stated result from ~\cite[Lemma 4.1]{ils} displays the contribution of the squares of primes, i.e. the term $P^{(2)}(f, h)$. This term can be included in the error term under the generalized Riemann hypothesis of $L(s, \mathrm{sym}^2 f)$, that we assume for other -- but similar -- purposes, see \cite[(4.23)]{ils}.
\end{rmk}

\subsection{Explicit formula for central values}

The connection between central values of $L$-functions, sums over primes and sums over zeros dates back to Selberg, and can be found in \cite[Proposition~1]{rs} in the case of quadratic Dirichlet characters. The proof carries on \textit{mutatis mutandis}.
\begin{prop}[Explicit formula for central values]
\label{prop:explicit-formula-central-values}
Assume that $L(\tfrac12, f)$ is nonzero. We have, for all $x \leqslant c(f)$,
\begin{equation}    
\label{eq:explicit-central-value}
\log L(\tfrac12, f) = P(f, x) - \tfrac12 \log \log x + O\Big( \frac{\log c(f)}{\log x} + \sum_{\gamma_f} \log(1 + (\gamma_f \log x)^{-2}) \Big),
\end{equation}
where we defined the sum over primes
\begin{equation}
\label{pfx}
P(f, x)=\sum_{\substack{p<x \\ p\nmid q}} \frac{a_f(p)}{p^{1/2}}.
\end{equation}
\end{prop}

Note that the term $-\tfrac12\log\log x$ is the expected mean of the logarithmic central values as predicted by Conjecture \ref{conj:keating-snaith} and stated in Theorem \ref{thm:thm}. This property reduces the study of central $L$-values to the study of their distribution around the mean, which is governed by the above sum over primes (studied in Theorem \ref{prop:moments}) and by the sum over zeros in the error term (studied in Section \ref{sec:conclusion}). 

\subsection{Trace formulas}

We introduce in this section quasi-orthogonality statements which will be central to understand harmonic averages of coefficients. Recall that we denote $H_k(q)$ an orthogonal basis of Hecke \textit{newforms} of level $q$ and introduce $B_k(q)$ an orthogonal basis of \textit{all} Hecke eigenforms of level $q$.

\subsubsection{Petersson trace formula} Consider the averages
\begin{equation}
    \Delta_q(m, n) := \sumh_{f \in B_k(q)} a_f(m) a_f(n).
\end{equation}
We then have the following quasi-orthogonality statement \cite[Proposition 2.1]{ils}.

\begin{prop}[Petersson trace formula]
\label{prop:ptf}
    For $m, n, q \geqslant 1$, we have
    \begin{equation}
        \Delta_q(m, n) = \delta_{m, n} + 2\pi i^{-k} \sum_{\substack{c \geqslant 1 \\ q \mid c}} \frac{S(m, n, c)}{c} J_{k-1}\left( \frac{4\pi \sqrt{mn}}{c}\right)
    \end{equation}
    where $\delta_{m, n}$ is the Kronecker delta symbol, $J_{k-1}$ is the J-Bessel function of order $k-1$, and $S(m, n, c)$ is the $\mathrm{GL}(2)$ Kloosterman sum, defined by 
    \begin{equation}
        S(m, n, c) := \sum_{a \in (\mathbb{Z}/c\mathbb{Z})^\times} e\left( \frac{am + \bar{a}n}{c}\right), 
    \end{equation}
    where $\bar{a}$ denotes the inverse of $a$ in $(\mathbb{Z}/c\mathbb{Z})^\times$.
\end{prop}

Define the average of coefficients over the \textit{newforms}, 
\begin{equation}
     \Delta^\star_q(m, n) := \sumh_{f \in H_k(q)} a_f(m) a_f(n).
\end{equation}

We have the following sieving result that relates averages over $B_k(q)$ and over $H_k(q)$, in other words allowing to sieve oldforms in, see \cite[Lemma 2.3]{bcl}.

\begin{lem}
\label{lem:sieve}
    Suppose that $m, n , q$ are positive integers with $(mn,q)=1$, and let $q = q_1q_2$ where $q_1$ is the largest factor of $q$ satisfying $p \mid q_1 \Leftrightarrow p^2 \mid q$ for all primes $p$. Then we have
    \begin{equation}
        \Delta^\star_q(m, n) = \sum_{\substack{q = L_1L_2d \\ L_1 \mid q_1 \\ L_2 \mid q_2}} \frac{\mu(L_1L_2)}{L_1L_2} \prod_{\substack{p \mid L_1 \\ p^2 \nmid d}} \left( 1-p^{-2}\right)^{-1}  \sum_{e \mid L_2^\infty} \frac{\Delta_d(me^2, n)}{e}.
    \end{equation}
\end{lem}

\begin{rmk}
Note that, because of the Möbius function $\mu(L_1L_2)$, we necessarily have $(L_2, d)=1$ and $(e, d) = 1$, which will be of much use later.
\end{rmk}

\subsubsection{Kuznetsov trace formula}
\label{subsubsec:ktf}

The spectral theory of automorphic forms is explained in various references such as \cite[Chapter 15]{ik}. We introduce the notations from \cite[Lemma 3.1]{bcl} to describe the three types of elements of the spectrum.
\begin{enumerate}
    \item[] (Modular forms) Let $B_\ell(q)$ be an orthogonal basis of the space of holomorphic cusp forms of weight $\ell$ and level $q$, which dimension is denoted $\theta_\ell(q)$. We can write $f_{j, \ell}$ for the elements of $B_\ell(q)$ and introduce their Fourier coefficients through the Fourier expansion 
\begin{equation}
    f_{j, \ell}(z) = \sum_{n \geqslant 1} \psi_{j, \ell}(n) (4\pi n)^{\ell / 2} e(nz).
\end{equation}
We say that $f$ is a Hecke eigenform if it is an eigenfunction of all the Hecke operators $T_n$ for $(n,q)=1$.  Note that for $(n,q)=1$, we have $a_{j, \ell}(n) a_{j, \ell}(1) = \sqrt{n}\psi_{j, \ell}(n)$.\medskip
    \item[] (Maass forms) Let $\lambda_j = \tfrac14 + \kappa_j^2$ be the eigenvalues of the hyperbolic Laplacian counted with multiplicities and in increasing order, in the space of cusp forms on $L^2(\Gamma_0(q) \backslash \mathbb{H})$. By convention, choose the sign of $\kappa_j$ such that $\kappa_j \geqslant 0$ when $\lambda_j \geqslant \tfrac14$ and $i \kappa_j > 0$ when $\lambda_j < \tfrac14$. For each positive $\lambda_j$, choose corresponding eigenvectors $u_j$ in such a way that the family~$(u_j)_j$ forms an orthonormal basis of the corresponding eigenspace, and define the associated Fourier coefficients~$\rho_j(m)$ by the Fourier expansion
\begin{equation}
    u_j(z) = \sum_{m\neq 0} \rho_j(m) W_{0, i\kappa_j}(4\pi|n| y)e(mx), 
\end{equation}
where $W_{0,it}(y) := (y/\pi)^{1/2}K_{it}(y/2)$ is a Whittaker function, and $K_{it}$ is the modified Bessel function of the second kind. We call $u$ a Hecke eigenform if it is an eigenfunction of all the Hecke operators~$T_n$ for $(n,q)=1$, and we then denote by $\lambda_u(n)$ the Hecke eigenvalue of $u$ for $T_n$. Writing~$\rho_u(n)$ as the Fourier coefficient, we have $\lambda_u(n) \rho_u(1) = \sqrt{n}\rho_u(n)$ when $(n,q)=1$. When $u$ is a newform, this holds for all $n \neq 0$ instead.  \medskip
    \item[] (Einsenstein series) Let $\mathfrak{c}$ be a cusp for $\Gamma_0(q)$. We define $\varphi_{\mathfrak{c}}(m,t)$ to be the $m$-th Fourier coefficient of the real-analytic Eisenstein series at $\tfrac12 + it$, i.e. by the Fourier expansion
\begin{equation}
    E_{\mathfrak{c}}(z, \tfrac12 + it) = \delta_{\mathfrak{c} = \infty} y^{\tfrac12 + it} + \varphi_{\mathfrak{c}}(0,t) y^{\tfrac12 - it} + \sum_{m \neq 0} \varphi_{\mathfrak{c}}(m,t) W_{0,it}(4\pi|n|y)e(mx).
\end{equation}
\end{enumerate}

\begin{prop}[Kuznetsov trace formula]
\label{prop:ktf}
    For $\phi : (0, \infty) \to \mathbb{C}$ a smooth and compactly supported function, and $m, n, q \geqslant 1$, we have 
    \begin{align}
        \sum_{\substack{c \geqslant 1 \\ q \mid c}} \frac{S(m, n, c)}{c} \phi\left( \frac{4\pi \sqrt{mn}}{c}\right) & = \sum_{\substack{\ell \in 2\mathbb{N}_+ \\ 1 \leqslant j \leqslant \theta_\ell(q)}} (l-1)! \sqrt{mn}\ \overline{\psi_{j, \ell}}(m) \psi_{j, \ell}(n) \phi_h(\ell) \\
        & \qquad + \sum_j \frac{\overline{\rho_j}(m) \rho_j(n) \sqrt{mn}}{\cosh \pi \kappa_j} \phi_+(\kappa_j)  \\
        & \qquad + \frac{1}{4\pi} \sum_{\mathfrak{c}} \int_{-\infty}^{+\infty} \frac{\sqrt{mn}}{\cosh \pi t} \overline{\varphi_{\mathfrak{c}}}(m, t) \varphi_{\mathfrak{c}}(n, t) \phi_+(t) dt, 
    \end{align}
    where the Bessel transforms are defined as by
    \begin{align}
        \phi_+(r) & := \frac{2\pi i}{\mathrm{sinh}(\pi r)} \int_0^\infty (J_{2ir}(\xi) - J_{-2ir}(\xi)) \phi(\xi) \frac{d\xi}{\xi} \\
        \phi_h(\ell) & := 4i^k \int_0^\infty J_{\ell-1}(\xi) \phi(\xi) \frac{d\xi}{\xi}
    \end{align}
    where $J_\ell$ is the J-Bessel function of the first kind.
\end{prop}

\section{Moments}
\label{subsec:moments}
\label{sec:moments}

By the explicit formula \eqref{eq:explicit-central-value}, a critical quantity to understand in order to control the distribution of the central values is the sums over primes $P(f,x)$, and this will be investigated by means of the moment method as in \cite{rs}. The following result is the fundamental tool to understand their distribution, and is an analogue of \cite[Theorem 3.1]{miller}.
\begin{thm}[Moment property]
\label{prop:moments}
We have, for all $\ell \geqslant 0$, 
\begin{equation}
\label{moments}
\frac{1}{N(Q)} \sum_{q \geqslant 1} \Psi\left( \frac{q}{Q} \right) \sumh_{f \in H_k(q)} P(f,x)^\ell  = (M_\ell + o(1)) (\log \log(x))^{\ell/2}
\end{equation}
where we introduced the $\ell$-th Gaussian moment
\begin{equation}
M_\ell = \frac{1}{\sqrt{2\pi}} \int_{\mathbb{R}} x^\ell e^{-x^2/2} dx = \frac{\ell!}{2^{ \ell/2} (\ell/2)!}.
\end{equation}
\end{thm}

\begin{rmk}
On average over the family, Theorem \ref{prop:moments} states that the moments of~$P(f,x)$, i.e. essentially the moments of $\log L(\tfrac12,f) + \tfrac12 \log\log x$ by the explicit formula \eqref{eq:explicit-central-value}, match the moments of the normal distribution, hence justifying the shape of Conjecture \ref{conj:keating-snaith} and of Theorem \ref{thm:thm}.
\end{rmk}

The remaining of this section as well as the following one constitute the proof of this result and of two corollaries.

\subsection{Sums over primes of coefficients}
\label{subsec:sums-over-primes}

We follow the strategy of \cite[Proposition 3]{rs} using the tools developed in \cite[Proposition 4.1]{miller}, adapting it to the specific sum over primes $P(f,x)$ arising in the explicit formula. After expanding the power $P(f,x)^\ell$ in 
\begin{equation}
\frac{1}{N(Q)} \sum_{q \geqslant 1} \Psi\left( \frac{q}{Q} \right) \sumh_{f \in H_k(q)} P(f, x)^\ell 
\end{equation}
and gathering together primes that are equal, we are reduced to study sums of the type
\begin{equation}
\label{eq:expanded-moment}
\frac{1}{N(Q)} \sum_{q \geqslant 1} \Psi\left( \frac{q}{Q} \right) \sumh_{f \in H_k(q)} \sum_{\substack{p_1, \ldots, p_l \leqslant x \\ p_i \nmid q \\ p_i \neq p_j}} \frac{a_f(p_1)^{\alpha_1} \cdots a_f(p_\ell)^{\alpha_\ell}}{p_1^{\alpha_1/2} \cdots p_\ell^{\alpha_\ell/2}}, 
\end{equation}
where the $\alpha_i$ are positive integers.
Inspired by the above expression of the expanded moment, introduce the notation, for any integer $\alpha \geqslant 1$ and prime number $p$, 
\begin{equation}
F(p, \alpha) := \frac{a_f(p)^\alpha}{p^{\alpha/2}}.
\end{equation}
By the expansion \eqref{eq:expanded-moment}, it is sufficient to study sums of products of such $F(p, \alpha)$.
We state in this section some first estimates for these quantities. Informally, the sum for higher powers $\alpha \geqslant 3$ will contribute negligibly, the contribution of the sum for powers $\alpha=2$ will display a precise equivalent by means of the Rankin-Selberg method and will determine the effective distribution in Theorem \ref{thm:thm}, and the sum for powers~$\alpha=1$ will be studied using the Perron formula and bounds on $L$-functions.

\begin{lem}[Large parts]
\label{lem:large-parts}
We have, for all $\alpha \geqslant 3$ and uniformly on $f \in H_k(q)$, 
\begin{equation}
\sum_{\substack{p \leqslant x \\ p \nmid q}} F(p,\alpha) \ll 1.
\end{equation}
\end{lem}

\begin{proof}
Using Deligne's bound $|a_f(p)| \leqslant 2$, the result follows since it reduces to the sum of $p^{-3/2}$ which converges absolutely.
\end{proof}

\begin{lem}[$2$-parts]
\label{lem:2-parts}
We have
\begin{equation}\label{2-parts-bound}
\sum_{\substack{p\leqslant x \\ p \nmid q}} F(p,2) \ll \log\log(x),
\end{equation}
where the implied constant is absolute.
\end{lem}

\begin{proof}
The bound \eqref{2-parts-bound} is immediate by Deligne's bound $|a_f(p)| \leqslant 2$ and using Mertens' estimate
$$ \sum_{p\leqslant x} \frac1p \ll \log\log(x). $$
\end{proof}

\begin{lem}[$1$-parts]
\label{lem:1-parts}
We have, for all $n \geqslant 1$ and all $x \leqslant q$, 
\begin{equation}
\sum_{\substack{p_1, \ldots, p_n \leqslant x \\ p_i \nmid q \\ p_i \neq p_j}} \prod_{i=1}^n F(p_i, 1) \ll \log(q)^{n + \varepsilon}.
\end{equation}
\end{lem}

\begin{proof}
The result \cite[Lemma 2.12]{miller} reads
\begin{equation}
\sum_{\substack{p \leqslant x \\ p \nmid q}} \left( b_f(p) := \frac{a_f(p)\log(p)}{p^{1/2}}\right) \ll \log(x)^{1+\varepsilon}\log(q)
\end{equation}
and by partial summation, we therefore deduce 
\begin{equation}
\sum_{\substack{p \leqslant x \\ p \nmid q}} \frac{a_f(p)}{p^{1/2}} = \sum_{\substack{p \leqslant x \\ p \nmid q}} \frac{b_f(p)}{\log p} \ll \sum_{\substack{p \leqslant x \\ p \nmid q}} \Bigg| \sum_{\substack{p' \leqslant p \\ p' \nmid q}} b_f(p')\Bigg| \frac{1}{p\log^2 p} \ll \log(q)^{1+\varepsilon} \sum_{\substack{p \leqslant x \\ p \nmid q}} \frac{1}{p\log p} \ll \log(q)^{1+\varepsilon}
\end{equation}
giving the desired result for $n=1$. We finish the proof inductively for $n \geqslant 1$, adding back the missing primes in order to get a genuine product, which incur an extra contribution made of higher powers, therefore of negligible size by the two above lemmas.
\end{proof}

\begin{rmk}
The proof of \cite[Lemma 2.12]{miller} boils down to using the Perron formula to relate the sought sum to $L'/L$, on which we have bounds that are enough for the result. 
Note that these "rough" bounds on the $1$-parts and $2$-parts will not be sufficient to bound the whole sum over the family, since the expected main term in Theorem \ref{prop:moments} is of size $\log \log x$ while the above bounds are about $\log x$ and $\log \log x$. The harmonic average (in the guise of trace formulas) or finer properties of $L$-functions will have to be fully exploited in order to get enough cancellations. These bounds will however be sufficient to address number of cases and remains fundamental in the proofs. 
\end{rmk}

The next paragraphs of this section are devoted to prove finer estimates for the $2$-parts (see Proposition \ref{lem:2-part-mt-single}) and to prepare the stage to estimating the $1$-parts.

\subsection{A Rankin bound}
\label{subsec:rankin-bound}

We need a finer and genuine asymptotics for the $2$-part, since it will ultimately contribute as the main term. 
We have the following statement, which is the standard Rankin bound with emphasis on the uniformity of the error term in the function $f$ --- and where we use the generalized Riemann hypothesis.
\begin{prop}
    \label{lem:2-part-mt-single}
    For all~$f \in H_k(q)$, for all $x < c(f)$, and assuming that $L(s,\mathrm{sym}^2 f)$ has no zeros in the rectangle $\{ z \ : \ \sigma_0 \leqslant \Re(z) \leqslant 1, \ |\Im(z) - t| \leqslant 3\}$ for a certain $1/2 \leqslant \sigma_0 < 1$, we have
\begin{equation}
     \sum_{\substack{p < x \\ p \nmid q}} \frac{\lambda_f(p)^2 }{p} = \log\log x + O(\log \log \log c(f)), 
\end{equation}
     where the implied constant is absolute.
\end{prop}

The remaining of Section \ref{subsec:rankin-bound} is dedicated to the proof of this result.
We start adding back the missing primes and use Deligne's bound $\lambda(p) \ll 1$ to control the incurred error.
We need the following estimate.

\begin{lem}
    For all $x \leqslant q$, we have
\begin{equation}\label{add_back}
    \sum_{\substack{p \le x \\ p\mid q}} \frac{1}{p} \ll \log\log\log(q),
\end{equation}
for $x\leqslant q$.
\end{lem}

\begin{proof}
Let $\omega(q)$ be the number of distinct prime divisors of $q$, and write the prime factorization of $q$ as $q=p_1^{\alpha_1}p_2^{\alpha_2} \cdots p_k^{\alpha_k}$, where each $p_n$ is a prime factor of $q$ ordered so that $p_1 < p_2 < \cdots < p_k$ and $k=\omega(q)$.
Then
\begin{equation}
    \log{q} = \sum_{n=1}^k \alpha_n \log{p_n} \geqslant k\log2 = (\log2)\omega(q)
\end{equation}
so that $\omega(q) \ll \log q$. Since $p_n \asymp n\log n$ by a classical estimate dating back to Chebyshev --- we more precisely know that $p_n \sim n \left(\log{n} + \log\log{n} - 1 \right)$ by Dusart \cite{dusart18} --- we deduce
\begin{equation}\label{p_omega_q}
p_{\omega(q)} \ll \omega(q)\log{\omega(q)} \ll \log(q)\log\log(q).
\end{equation}
We therefore deduce
\begin{equation}
\sum_{\substack{p \le x \\ p\mid q}} \dfrac1p \ll \sum_{n=1}^{\omega(q)} \dfrac1{p_n} \sim \log\log(p_{\omega(q)}) \ll \log\log\log(q)
\end{equation}
by Mertens estimate on the sum over reciprocals of primes.
\end{proof}

The problem therefore reduces to estimating the complete sum over primes
\begin{equation}
    \sum_{\substack{p \leqslant x \\ p \nmid q}} \frac{\lambda(p)^2 }{p} \rightsquigarrow \sum_{\substack{p \leqslant x}} \frac{\lambda(p)^2 }{p}.
\end{equation}

We use the Hecke relation $\lambda(p)^2 = \lambda(p^2) + 1$, and use Mertens estimate
\begin{equation}
    \sum_{p \leqslant x} \frac{1}{p} = \log \log x + O(1),
\end{equation}
with an absolute error term. Therefore, for $x \leqslant q$, we have
\begin{equation}\label{eq:sum_over_p}
    \sum_{\substack{p \leqslant x \\ p \nmid q}} \frac{\lambda(p)^2}{p} = \sum_{\substack{p \leqslant x}} \frac{\lambda(p^2) }{p} + \log\log x + O(\log\log\log(q)), 
\end{equation} 
which already displays the main term in Theorem \ref{thm:thm}, so that we are reduced to study the sum of coefficients $\lambda_f(p^2)$. Such a sum may be completed into the sum of coefficients of squares of prime powers at the cost of a uniformly bounded error term, and the sum thus obtained has size determined by $L(1, \mathrm{sym}^2 f)$ \textit{via} the Perron formula, see \cite[Lemma 2.11]{miller}. It is therefore sufficient to bound L-values, and we follow the same strategy as in \cite[Corollary 4.4]{cm} to do so. We state these bounds for general $L$-functions $\log L(s, f)$.

\begin{lem}
\label{lem:l-bound}
    Let $s = \sigma + it$ with $\sigma > \tfrac12$ and $|t| \leqslant 3c(f)$. Let $\sigma_0 \in (1/2, \sigma)$.
    Suppose $L(s,f)$ has no zeros in the rectangle $\{ z \ : \ \sigma_0 \leqslant \Re(z) \leqslant 1, \ |\Im(z) - t| \leqslant 3\}$. Then we have, uniformly in $f$, 
    \begin{equation}
        \log L(s, f) \ll \frac{\log c(f)}{\sigma-\sigma_0}.
    \end{equation}
\end{lem}

\begin{proof}
    For $\sigma > 2$, then we have $\log L(s,f) \ll 1$ uniformly in $f$, by Deligne's bound $a_f(n) \ll 1$ and by absolute convergence. 

    Assume $\sigma \leqslant 2$ and follow the strategy of Granville-Soundararajan \cite{gs}; they prove the analogous result for Dirichlet $L$-functions, but the result is general as used for instance in \cite[Lemma 2.3]{cm}, unproven there. Consider circles of center $2+it$ and radii $r := 2-\sigma$, so that they pass through $s$. 
    On the larger circle of radius $R := 2-\sigma_0$, we have for all $z$ on the circle, 
    \begin{equation}
    \label{xxx}
        L(z,f) \ll c(f) |z| \ll c(f)^2.
    \end{equation}
This follows from convexity bounds \cite[Lemma 6.7 and Theorem 6.8]{steuding}, which states that $L$-values in the critical strip are bounded by $L(z,f) \ll (|z|^2 c(f)^2)^{1-\sigma}$ for all $\sigma \in (0,1)$, uniformly in $f$. We get the bound \eqref{xxx}, for $\sigma \in (1/2, 1)$, the worst case being $\sigma = 1/2$.

For $s$ in the rectangle $\{ z \ : \ \sigma_0 \leqslant \Re(z) \leqslant 1, \ |\Im(z) - t| \leqslant 3\}$, we use the Borel-Carath\'eodory theorem to obtain
\begin{align*}
    |\log L(s,f)| & \leqslant \frac{2r}{R-r} \max_{|z-(2+it)| = R} \Re \log L(z,f) + \frac{R+r}{R-r} |\log L(2+it, f)| \ll \frac{\log c(f)}{\sigma - \sigma_0} + \frac{1}{\sigma-\sigma_0}
\end{align*}
which is as desired.
\end{proof}

Assuming strong zero-free regions, typically implied by the generalized Riemann hypothesis, we can approximate $L$-functions by short sums of coefficients, with an error term depending on the chosen length. The following lemma makes it precise.
\begin{lem}
    Let $s= \sigma + it$ with $\sigma > 1/2$ and $|t| \leqslant 2c(f)$. Let $y \geqslant 2$ a real parameter and $\sigma_0 \in (1/2, \sigma)$. Suppose there are no zeros of $L(z, f)$ in the rectangle $\{z \ : \ \sigma_0 \leqslant \Re(z) \leqslant 1, \ |\Im(z) - t| \leqslant y+ 3 \}$. Let $\sigma_1 = \min(\tfrac12(\sigma+\sigma_0), \sigma_0 + 1/\log(y))$. We then have
    \begin{equation}
        \log L(s,f) = \sum_{n=2}^y \frac{\Lambda(n) a_f(n)}{n^s \log n} + O\left( \frac{\log c(f)}{(\sigma_1 - \sigma_0)^2} y^{\sigma_1-\sigma} \right), 
    \end{equation}
    where the implied constant is independent of $f$.
\end{lem}

\begin{proof}
    By the truncated Perron formula given in \cite[Corollary 5.3]{mv} or \cite[Example 4.4.15]{murty}, we can express the above short sum by a vertical integral of the $L$-function. For $c = 1-\sigma + 1/\log y$, since $L(s,f)$ is entire,  we have
    \begin{align}
        \frac{1}{2\pi i} \int_{c-iy}^{c+iy} \log L(s+w, f) \frac{y^w}{w} dw & = \sum_{n = 2}^y \frac{\Lambda(n) a_f(n)}{n^s \log n} + O\left( \frac{1}{y} \sum_{n \geqslant 1} \frac{y^c}{n^{\sigma + c}} \frac{1}{|\log(y/n)|} \right) \\
        & = \sum_{n = 2}^y \frac{\Lambda(n) a_f(n)}{n^s \log n} + O\left( y^{-\sigma} \log y \right).
    \end{align}

    We can on the other hand move the integration line from $\Re = c$ to $\Re = \sigma_1 - \sigma < 0$. We exactly assumed that there are no zero in the rectangle thus crossed, i.e. we pick no singularities, except $w = 0$ where there is a simple pole with residue $\log L(s,f)$. The remaining integrals after picking up this $L$-value are three segments of the form
    \begin{equation}
        \int \log L(s+w, f) \frac{y^w}{w}dw, 
    \end{equation}
    on the segments $[c\pm iy, \sigma_1-\sigma \pm iy]$ and $[\sigma_1 - \sigma - iy, \sigma_1 - \sigma + iy]$.
    By the previous Lemma \ref{lem:l-bound} and the assumptions on $\Re(w) \geqslant \sigma_1-\sigma$, these integrals are bounded by
    \begin{equation}
        \ll \int \frac{\log c(f)}{\sigma + \Re(w) - \sigma_0} \frac{y^w}{w} dw \ll \frac{\log c(f)}{(\sigma_1 - \sigma_0)^2} y^{\sigma_1 - \sigma}. 
    \end{equation}
    This ends the proof of the lemma.
\end{proof}

We can now instantiate the above lemma with suitable choices of variables to obtained the desired bound: 
\begin{lem}
\label{lem:L-bound}
    Let $\eta > 2(\log c(f))^{-1}$. Assume that $L(s,f)$ has no zeros in the rectangle $\Re(s) \in [1-\eta, 1]$ and $|\Im(s)| \leqslant \log^{10/\eta}\! c(f)$ --- this is for instance true with any $0<\eta<1/2$ when assuming the Generalized Riemann Hypothesis. Then
    \begin{equation}
        \log L(s,f) \ll \log\log\log c(f)
    \end{equation}
    uniformly for $\Re(s) \geqslant 1-1/\log\log c(f)$ and $|\Im(s)| \leqslant \log^{10}c(f)$.
\end{lem}

\begin{proof}
    Apply the above with $\sigma_0 = 1-\eta$, $\sigma = \Re(s) \geqslant 1-1/\log\log  c(f)$, and $\sigma - \sigma_0 \geqslant \sigma - \sigma_1 \geqslant \eta/2$. Choose the value $y = \log^{10/\eta}\! c(f)$. Use Deligne's bounds $|\Lambda_f(p^\alpha)| \ll \log p$ to remove higher powers of primes in the sum, corresponding to a bounded contribution, and get
    \begin{equation}
        |\log L(s,f)| \leqslant \left| \sum_{p=2}^y \frac{\Lambda(p) a_f(p)}{p^s \log p} \right| + O(1) \ll \sum_{p=2}^{\log^{10/\eta}\! c(f)} \frac{1}{p^{1-1/\log\log  c(f)}} \ll \log\log\log c(f).
    \end{equation}
    This is the claimed result.
\end{proof}

We obtain Proposition \ref{lem:2-part-mt-single} by using  Perron formula to relate the sum of $\lambda_f(p^2)$ to~$L(1, \mathrm{sym}^2 f)$, and Lemma \ref{lem:L-bound} for $\mathrm{sym}^2 f$ at $s=1$ to bound this $L$-value.
By averaging Proposition \ref{lem:2-part-mt-single} over the family, we obtain the following estimate for the $2$-parts.

\begin{coro}
\label{lem:2-part-mt}
Assume the Generalized Riemann Hypothesis for $L(s, \mathrm{sym}^2 f)$.
For all $Q^\delta \ll x \ll Q$, we have 
\begin{equation}
    \frac{1}{N(Q)} \sum_{q \geqslant 1} \Psi\left(\frac{q}{Q}\right) \sumh_{f \in H_k(q)} \sum_{\substack{p \leqslant x \\ p\nmid q}} F(p,2) = \log\log x + O(\log\log\log x),
\end{equation}
where the implied constant is absolute.
\end{coro}

\subsection{Case splitting and reduction to powers one}
\label{subsec:splitting}

Recalling the definition of $F(p,a)$, the expression \eqref{eq:expanded-moment} splits into sums of the type
\begin{equation}
\label{s}
\frac{1}{N(Q)} \sum_{q \geqslant 1} \Psi\left( \frac{q}{Q} \right) \sumh_{f \in H_k(q)} \sum_{\substack{p_1, \ldots, p_\ell \leqslant x \\ p_i\nmid q \\ p_i \neq p_j}} \prod_{i=1}^\ell F(p_i, \alpha_i)
\end{equation}
so that it is sufficient to study these.
We split into different cases according to the number of conspiring primes (i.e. the size of the powers $\alpha_i$), and use the lemmas established in Sections \ref{subsec:sums-over-primes} and \ref{subsec:rankin-bound} to treat each part. Define the following cases:
\begin{itemize}
\item[] Case A --- each power $\alpha_i$ is $2$;
\item[] Case B --- each power is at least $2$, at least one is larger;
\item[] Case C --- at least one power is $1$, but not all;
\item[] Case D --- each power is $1$.
\end{itemize}
To prove Theorem \ref{prop:moments}, we proceed by induction on the number of terms $\ell$ in the product. The remainder of Section \ref{subsec:splitting} is dedicated to treating the cases A, B and C or to reduce them to case~D, which will be addressed in Section \ref{sec:only-ones}.

\subsubsection*{Case A: each power is $2$}

By the estimate of Lemma \ref{lem:2-part-mt}, we have
\begin{equation}\label{eq:Case1-base}
\sum_{\substack{p \leqslant x \\ p\nmid q}} F(p,2) = \sum_{\substack{p \leqslant x \\ p\nmid q}} \frac{a_f(p)^2}{p} = \log\log(x) + O(\log\log\log c(f)), 
\end{equation}
which corresponds to the situation $\ell = 1$ in Case A.
Let $\ell \geqslant 1$ and assume inductively that for all $l=1,2,\ldots,\ell$, we have
\begin{equation}\label{eq:Case1-induct}
\frac{1}{N(Q)} \sum_{q \geqslant 1} \Psi\left(\frac{q}{Q}\right) \sumh_{f \in H_k(q)}  \sum_{\substack{p_1, \ldots, p_l \leqslant x \\ p_i\nmid q \\ p_i \neq p_j}} \prod_{i=1}^l F(p_i,2) =  (\log\log(x))^l + o((\log\log(x))^l).
\end{equation}
We will prove the corresponding property for $\ell + 1$.
Adding back the missing primes in order to complete one of the sums over the primes $ p = p_{\ell  +1}$, we get
\begin{align}
& \sum_{\substack{p_1, \ldots, p_{\ell+1} \leqslant x \\ p_i\nmid q \\ p_i \neq p_j}} \prod_{i=1}^{\ell+1} F(p_i,2)
= \sum_{\substack{p_1, \ldots, p_{\ell} \leqslant x \\ p_i\nmid q \\ p_i \neq p_j}} \prod_{i=1}^{\ell} F(p_i,2) \sum_{\substack{p\leqslant x \\ p\nmid q \\ p\neq p_i,\, 1\le i\le\ell}} F(p,2) \notag\\
& \qquad = \Bigg(\sum_{\substack{p_1, \ldots, p_{\ell} \leqslant x \\ p_i\nmid q \\ p_i \neq p_j}} \prod_{i=1}^{\ell} F(p_i,2)\Bigg) \Bigg(\sum_{\substack{p\leqslant x \\ p\nmid q}} F(p,2)\Bigg) - \ell \Bigg(\sum_{\substack{p_2, \ldots, p_{\ell} \leqslant x \\ p_i\nmid q \\ p_i \neq p_j}} \prod_{\substack{i=2}}^{\ell} F(p_i,2)\Bigg)\Bigg(\sum_{\substack{p\leqslant x \\ p\nmid q}} F(p,4)\Bigg). \label{eq:power2_add_back}
\end{align}
 By Lemma \ref{lem:large-parts}, the sum including $F(p,4)$ in \eqref{eq:power2_add_back} is uniformly bounded, so that the rightmost term in \eqref{eq:power2_add_back} falls into the induction hypothesis and is bounded by $(\log\log(x))^\ell = o(\log\log(x))^{\ell+1}$.   In the left term of \eqref{eq:power2_add_back}, we average over the family and use Hölder inequality as in \cite{miller}, as well as induction to conclude it is equivalent to $\log \log(x)^{\ell + 1}$, proving the case for $\ell + 1$ and finishing the proof by induction.
Thus
\begin{align}
\frac{1}{N(Q)} \sum_{q \geqslant 1} \Psi\left(\frac{q}{Q}\right) \sumh_{f \in H_k(q)} \sum_{\substack{p_1, \ldots, p_{\ell+1} \leqslant x \\ p_i\nmid q \\ p_i \neq p_j}} \prod_{i=1}^{\ell} F(p_i,2) \sim (\log\log(x))^{\ell}
\end{align}
for any $\ell\in\mathbb{N}^\star$ by induction. 

It remains to estimate the contribution of these terms corresponding to the powers $\alpha_i = 2$ in the~$k$-th moment of $P(f,x)$. The $k$ primes can be paired into $k/2$ such squares, giving the contribution of$(\log\log(x))^{k/2}$ in Theorem \ref{prop:moments}. The number of such pairings can be obtained as follows: select~$k/2$ primes, then pair each of them with one of the $k/2$ remaining primes (which has~$(k/2)!$ possibilities), and notice that every pairing has been counted twice because of the possible swaps~$p_i \leftrightarrow p_j$ (which are $2^{k/2}$), so that the total number of such contributions falling into Case~A is
\begin{equation}
    \binom{k}{k/2} \frac{(k/2)!}{2^{k/2}} = \frac{k!}{(k/2)! 2^{k/2}} = \frac{1}{\sqrt{2\pi}} \int_{\mathbb{R}} x^k e^{-x^2/2} dx, 
\end{equation}
and we exactly recover the $k$-th moment of the normal law, justifying the main contribution in Theorem \ref{prop:moments}. 

\subsubsection*{Case B: each power is at least $2$, at least one being larger than $2$}

By Lemma \ref{lem:large-parts}, we have
\begin{equation}\label{eq:Case2-1}
\sum_{\substack{p \leqslant x \\ p\nmid q}} F(p,\alpha) = O(1)
\end{equation}
whenever $\alpha\ge3$, the underlying constant being absolute. Equation \eqref{eq:Case2-1} immediately implies that this holds on average over the family, in particular
$$ \frac{1}{N(Q)} \sum_{q \geqslant 1} \Psi\left(\frac{q}{Q}\right) \sumh_{f \in H_k(q)} \sum_{\substack{p\leqslant x \\ p\nmid q}} F(p,\alpha) = o\left(\log\log(x)\right). $$

Let $\ell \geqslant 1$ and assume inductively that for all $l=1,2,\ldots,\ell$, we have
\begin{equation}\label{eq:Case2-induct}
\sum_{\substack{p_1, \ldots, p_l \leqslant x \\ p_i\nmid q \\ p_i \neq p_j}} \prod_{i=1}^l F(p_i, \alpha_i) = o((\log\log(x))^l),
\end{equation}
where $\alpha_i\ge2$ for all $i=1,2,\ldots,l$, and there is at least one $j\in\{1,2,\ldots,l\}$ such that $\alpha_j\ge3$. We address the case of $\ell+1$ factors. Reordering such that at least one $j\in\{1,2,\ldots,\ell\}$ satisfies $\alpha_j\ge3$, and adding the missing primes in order to complete the sum over primes $p_{\ell + 1}$, we can write
\begin{align}
& \sum_{\substack{p_1, \ldots, p_{\ell+1} \leqslant x \\ p_i\nmid q \\ p_i \neq p_j}} \prod_{i=1}^{\ell+1} F(p_i, \alpha_i)
= \sum_{\substack{p_1, \ldots, p_{\ell} \leqslant x \\ p_i\nmid q \\ p_i \neq p_j}} \prod_{i=1}^{\ell} F(p_i, \alpha_i) \sum_{\substack{p_{\ell+1} \leqslant x \\ p_{\ell+1} \nmid q \\ p_{\ell+1} \neq p_i,\, 1\le i\le\ell}} F(p_{\ell+1},\alpha_{\ell+1}), \notag\\
&\qquad = \Bigg(\sum_{\substack{p_1, \ldots, p_{\ell} \leqslant x \\ p_i\nmid q \\ p_i \neq p_j}} \prod_{i=1}^{\ell} F(p_i, \alpha_i)\Bigg) \Bigg(\sum_{\substack{p_{\ell+1} \leqslant x \\ p_{\ell+1} \nmid q}} F(p_{\ell+1} ,\alpha_{\ell+1})\Bigg) \label{eq:caseB_first}\\
&\qquad \qquad- \sum_{m=1}^\ell \Bigg(\sum_{\substack{p_1, \ldots, p_{m-1}, p_{m+1}, \ldots, p_{\ell} \leqslant x \\ p_i\nmid q \\ p_i \neq p_j}} \prod_{\substack{i=1 \\ i\neq m}}^{\ell} F(p_i, \alpha_i)\Bigg)\Bigg(\sum_{\substack{p_{\ell+1} \leqslant x \\ p_{\ell+1} \nmid q}} F(p_{\ell+1},\alpha_{\ell+1}+\alpha_m)\Bigg), \label{eq:caseB_second}
\end{align}
Examining \eqref{eq:caseB_second}, since $\alpha_{\ell+1} + \alpha_m \geqslant 3$, we can use the bound given by Lemma \ref{lem:large-parts} to conclude that the corresponding sum is uniformly bounded; the induction hypothesis therefore applies ($\alpha_j \geqslant 3$) to the other factor in \eqref{eq:caseB_second} and allows to conclude.  Examining \eqref{eq:caseB_first}, since $\alpha_{\ell+1} \geqslant~2$, we have that the sum over $p_{\ell+1}$ is uniformly bounded by $\log\log x$ by Lemma \ref{lem:2-parts}; the induction hypothesis therefore applies to the first factor in \eqref{eq:caseB_first} and allows to conclude that it is $o(\log \log (x))^{\ell/2}$, so that the whole expression is $o(\log \log(x)^{\ell+1})$.
This concludes the induction for Case B, and  we have
\begin{align*}
\frac{1}{N(Q)} \sum_{q \geqslant 1} \Psi\left(\frac{q}{Q}\right) \sumh_{f \in H_k(q)} \sum_{\substack{p_1, \ldots, p_{\ell} \leqslant x \\ p_i\nmid q \\ p_i \neq p_j}} \prod_{i=1}^{\ell} F(p_i, \alpha_i) = o\left((\log\log(x))^{\ell}\right)
\end{align*}
for any $\ell\in\mathbb{N}^\star$.

\subsubsection*{Case C: at least one power is $1$, but not all}
We now reduce Case C to Case D appealing to induction. Let $\ell \geqslant 2$. We assume inductively that, for all $l=2,\ldots,\ell$ such that at least one power $\alpha_i \geqslant 2$, we have
\begin{equation}\label{eq:Case4_induct}
\begin{aligned}
&\frac{1}{N(Q)} \sum_{q \geqslant 1} \Psi\left(\frac{q}{Q}\right) \sumh_{f \in H_k(q)}
\sum_{\substack{p_1, \ldots, p_l \leqslant x \\ p_i\nmid q \\ p_i \neq p_j}} \prod_{i=1}^l F(p_i, \alpha_i) \\
&\quad\ll \frac{(\log\log(x))^{l-n_l}}{N(Q)} \Bigg|\sum_{q \geqslant 1} \Psi\left(\frac{q}{Q}\right) \sumh_{f \in H_k(q)} \sum_{\substack{p_1, \ldots, p_{n_l} \leqslant x \\ p_i\nmid q \\ p_i \neq p_j}} \prod_{i=1}^{n_l} F(p_i,1)\Bigg|.
\end{aligned}
\end{equation}
where $n_l := \#\{j=1,2,\ldots,l \mid \alpha_j=1\}$ and we ordered the $\alpha_i$'s so that $\alpha_1 = \alpha_2 = \cdots = \alpha_{n_l} = 1$ and~$\alpha_{n_l+1}, \alpha_{n_l+2}, \ldots, \alpha_l \geqslant 2$.
Adding the missing primes to complete the sum over $p_{\ell+1}$, we have
\begin{align}
& \sum_{\substack{p_1, \ldots, p_{\ell+1} \leqslant x \\ p_i\nmid q \\ p_i \neq p_j}} \prod_{i=1}^{\ell+1} F(p_i, \alpha_i)
= \sum_{\substack{p_1, \ldots, p_{\ell} \leqslant x \\ p_i\nmid q \\ p_i \neq p_j}} \prod_{i=1}^{\ell} F(p_i, \alpha_i) \sum_{\substack{p_{\ell+1}\leqslant x \\ p_{\ell+1} \nmid q \\ p_{\ell+1} \neq p_i,\, 1\le i\le\ell}} F(p_{\ell+1},\alpha_{\ell+1}) \notag\\
&\qquad = \Bigg(\sum_{\substack{p_1, \ldots, p_{\ell} \leqslant x \\ p_i\nmid q \\ p_i \neq p_j}} \prod_{i=1}^{\ell} F(p_i, \alpha_i)\Bigg) \Bigg(\sum_{\substack{p_{\ell+1}\leqslant x \\ p_{\ell+1}\nmid q}} F(p_{\ell+1},\alpha_{\ell+1})\Bigg) \label{case4_1}\\
&\qquad \qquad- \sum_{m=1}^{\ell} \Bigg(\sum_{\substack{p_1, \ldots, p_{m-1}, p_{m+1}, \ldots, p_{\ell} \leqslant x \\ p_i\nmid q \\ p_i \neq p_j}} \prod_{\substack{i=1 \\ i\neq m}}^{\ell} F(p_i, \alpha_i)\Bigg) \Bigg(\sum_{\substack{p_{\ell+1}\leqslant x \\ p_{\ell+1}\nmid q}} F(p_{\ell+1},\alpha_{\ell+1}+\alpha_m)\Bigg) \label{case4_2}
\end{align}
The sum of $F(p_{\ell+1},\alpha_{\ell+1})$ in \eqref{case4_1} is uniformly dominated by $\log\log(x)$ by Lemma \ref{lem:2-parts} since $\alpha_{\ell+1} \geqslant 2$. 
The sum of $F(p_{\ell+1},\alpha_{\ell+1}+\alpha_m)$ in \eqref{case4_2} is uniformly bounded by Lemma \ref{lem:large-parts} since $\alpha_{\ell+1} + \alpha_m, \geqslant 3$.
The multiple sums of $F(p_i, \alpha_i)$ in \eqref{case4_1} and \eqref{case4_2} either contain a term $\alpha_i\ge2$, in which case we appeal to the induction hypothesis, or only contain terms $\alpha_i=1$, which is Case D.
This completes the induction reducing Case C to Case D.

\section{Case D and consequences}
\label{sec:only-ones}

The above sections reduced the proof of Theorem \ref{prop:moments} to the proof of case D, where each power is~$\alpha_i = 1$. We treat this case and deduce important consequences from this result.

\subsection{First moment of coefficients}
This case is the most difficult and requires to make use of the harmonic sum over the family, i.e. trace formulas. We have to prove the following bound on the $1$-parts to prove they contribute as an error term:
\begin{prop}
\label{prop:first-moments}
For all $\ell \geqslant 1$, we have
\begin{equation}
\label{eq:case5}
\frac{1}{N(Q)} \sum_{q \geqslant 1} \Psi\left(\frac{q}{Q}\right) \sumh_{f \in H_k(q)} \sum_{\substack{p_1, \ldots, p_{\ell} \leqslant x \\ p_i\nmid q \\ p_i \neq p_j}} \prod_{i=1}^\ell F(p_i, 1) = o\left((\log \log(x))^{\ell/2}\right).
\end{equation}
\end{prop}
The techniques used in~\cite{rs} for quadratic twists of elliptic curves do not apply in our case, mainly because they have complete multiplicativity of the characters and they could use Poisson summation formula for the sum of characters. We closely follow the strategy of \cite{bcl}, who proved the analogous result with different weights for a single prime, and extend inductively as in \cite{miller}. The rest of this section is dedicated to the proof of this result.

\begin{proof}
Using Hecke multiplicativity for modular coefficients, we rewrite $a_f(p_1) \cdots a_f(p_\ell) = a_f(p_1 \cdots p_\ell)$ since the primes are different. We are therefore reduced to study the sum
\begin{equation}
\Sigma := \frac{1}{N(Q)} \sum_{q \geqslant 1} \Psi\left(\frac{q}{Q}\right) \sumh_{f \in H_k(q)} \sum_{\substack{p_1, \ldots, p_{\ell} \leqslant x \\ p_i\nmid q \\ p_i \neq p_j}} \frac{a_f(p_1 \cdots p_\ell)}{\sqrt{p_1 \cdots p_\ell}}.
\end{equation}

The harmonic sum has to be exploited by means of trace formulas. However, trace formulas, viz. the Petersson trace formula here, can be used for sums over \textit{all} the modular forms of a given weight and level, not only over newforms as in \eqref{eq:case5}. It is therefore necessary to add back the missing oldforms in the sum. Using Lemma \ref{lem:sieve} to do so and swapping summations in \eqref{eq:case5}, we obtain
\begin{equation*}
    \Sigma = \frac{1}{N(Q)} \sum_{q \geqslant 1} \Psi\left(\frac{q}{Q}\right) \sum_{\substack{p_1, \ldots, p_\ell \leqslant x \\ p_i \neq p_j\\ p_i \nmid q}}  \frac{1}{\sqrt{p_1 \cdots p_\ell}} \sum_{\substack{q = L_1L_2d \\ L_1 \mid q_1 \\ L_2 \mid q_2}} \frac{\mu(L_1L_2)}{L_1L_2} \prod_{\substack{p \mid L_1 \\ p^2 \nmid d}} (1-p^{-2})^{-1} \sum_{e \mid L_2^\infty} \frac{\Delta_d(e^2, p_1 \cdots p_\ell)}{e}.
\end{equation*}

We truncate the summations over $L_1, L_2$ and over $e$. Consider the tail of the sum, for $L_1L_2 > L_0$, use the trivial estimate given in Lemma \ref{lem:1-parts} to bound the tail of $\Sigma$ and using that $N(Q) \asymp Q$ by standard dimension formulas, to get
\begin{align}
    & \frac{1}{Q} \sum_{L_1L_2 > L_0} \frac{1}{L_1L_2} \Psi\left( \frac{L_1L_2d}{Q}\right) \sum_{e \mid L_2^\infty} \frac{\tau(e^2)}{e} \left|\ \ \sumh_{f \in B_k(q)} \sum_{p_1, \ldots,p_\ell \leqslant x} \frac{a_f(p_1 \cdots p_\ell)}{\sqrt{p_1 \cdots p_\ell}} \right| \\
    & \quad \ll \frac{\log Q}{Q} \sum_{q \ll Q/L_0} \frac{\tau(L_2)}{L_1L_2} \Psi\left( \frac{L_1L_2d}{Q}\right) \ll \frac{\log Q}{Q} \sum_{q \ll Q/L_0} \log(Q)^\ell \ll \frac{\log(Q)^{\ell + 1}}{L_0}
\end{align}
so that we get an error term of constant size for any $L_0 > \log(Q)^{\ell + 1}$. We can bound roughly the sum for $L_1L_2 < L_0$ but $e>E$ and get an error of constant size so long as $E$ is not less than a power of $\log(Q)$. We therefore can restrict the sums from now on to these ranges $L_1L_2 < L_0$ and~$e < E$ for $E$ a certain power of $\log(Q)$, as in \cite{bcl}.
As in the above sections, we can add the sum over primes dividing the level, i.e. $p_i \mid q$, at a cost $\log\log\log q$, therefore getting a saving even without exploiting the summation over the family. We are therefore reduced to deal with $\Sigma = \Sigma_2 + O(1)$ where
\begin{equation}
    \Sigma_2 := \frac{1}{N(Q)} \sum_{q \geqslant 1} \Psi\left(\frac{q}{Q}\right) \sum_{p_i \neq p_j} \frac{1}{\sqrt{p_1 \cdots p_\ell}} \sum_{\substack{q = L_1L_2d \\ L_1L_2 < L_0 \\ L_1 \mid q_1 \\ L_2 \mid q_2}} \frac{\mu(L_1L_2)}{L_1L_2} \prod_{\substack{p \mid L_1 \\ p^2 \nmid d}} (1-p^{-2})^{-1} \sum_{\substack{e \mid L_2^\infty \\ e < E}} \frac{\Delta_d(e^2, p_1 \cdots p_\ell)}{e}.
\end{equation}

We now perform a more precise arithmetic parametrization of the summation, following \cite{bcl}. Recalling that $q = L_1L_2d$, we replace the conditions $L_i \mid q_i$ by $L_1 \mid d$, $(L_2, d)=1$, and~$d = L_1m$ as in \cite[Remark following Lemma 2.3]{bcl}. This in particular implies
\begin{equation}
    \prod_{\substack{p_1 \mid L_1 \\ p_1^2 \nmid d}} (1-p^{-2})^{-1} = \prod_{p \mid L_1} (1-p^{-2})^{-1} \prod_{r \mid (L_1, m)} \frac{\mu(r)}{r^2}.
\end{equation}

We therefore have $q = L_1^2L_2m$ where $(m, L_2)=1$, and write $m = rn$ since $r \mid m$. We altogether have 
\begin{align}
    \Sigma_2 & = 
     \frac{1}{N(Q)} \sum_{q \geqslant 1} \Psi\left(\frac{q}{Q}\right) \sum_{p_i \neq p_j} \frac{1}{\sqrt{p_1 \cdots p_\ell}} \sum_{\substack{q = L_1^2L_2rn \\ L_1L_2 < L_0 \\ (L_1rn, L_2) = 1}} \frac{\mu(L_1L_2)}{L_1L_2} \\
     & \qquad \times \prod_{p_1 \mid L_1} (1-p^{-2})^{-1} \prod_{r \mid (L_1, rn)} \frac{\mu(r)}{r^2} \sum_{\substack{e \mid L_2^\infty \\ e < E}} \frac{\Delta_d(e^2, p_1 \cdots p_\ell)}{e}.
\end{align}

By Petersson trace formula from Proposition \ref{prop:ptf}, noting that $e^2 \neq p_1 \cdots p_\ell$ since the primes are assumed to be all different, we get
\begin{equation*}
\Delta_d(e^2, p_1 \cdots p_\ell) = 2\pi i^{-k} \sum_{c \geqslant 1} \frac{S(e^2, p_1 \cdots p_\ell, cL_1 r n)}{cL_1 r n}J_{k-1}\left( \frac{4\pi \sqrt{e^2p_1 \cdots p_\ell}}{cL_1 r n}\right).
\end{equation*}
Using M\"obius inversion to detect the condition $(L_2, n)=1$, we can rephrase $\Sigma_2$ as
\begin{align*}
    \Sigma_2 & = \frac{2\pi i^{-k} }{N(Q)} \sum_{q \geqslant 1} \Psi\left( \frac{q}{Q}\right) \sum_{\substack{L_1L_2 < L_0 \\ (L_1, L_2)=1}} \frac{\mu(L_1L_2)}{L_1L_2} \prod_{p \mid L_1} (1-p_1^{-2})^{-1} \sum_{\substack{p_1, \ldots, p_\ell \\ p_i \neq p_j}} \frac{1}{\sqrt{p_1 \cdots p_\ell}} \sum_{\substack{e \mid L_2^\infty \\ e < E}} \frac{1}{e} \\
    & \quad \times \sum_{c \geqslant 1} \sum_{n \geqslant 1} \sum_{d \mid L_2} \mu(d) \sum_{r \mid L_1} \frac{\mu(r)}{r^2} \frac{S(e^2, p_1 \cdots p_\ell, cL_1 dnr)}{cL_1 nrd} \Psi\left( \frac{L_1^2L_2 nrd}{Q}\right) J_{k-1} \left( \frac{4\pi\sqrt{e^2 p_1 \cdots p_\ell}}{cL_1nrd} \right).
\end{align*}

Noting $\mathfrak{m} := cL_1dnr \equiv 0$ modulo $cL_1rd$, we get that the sum over $n$ is
\begin{equation}
    \sum_{\mathfrak{m} \equiv 0 (cL_1 d r)} \frac{S(e^2, p_1 \cdots p_\ell, \mathfrak{m})}{\mathfrak{m}} f\left( \frac{4\pi\sqrt{e^2p_1 \cdots p_\ell}}{\mathfrak{m}}\right),
\end{equation}
where we introduced
\begin{equation}
    f(\xi) := \Psi\left(\frac{2\pi \sqrt{e^2p_1 \cdots p_\ell} L_1L_2}{cQ\xi}\right) J_{k-1}(\xi).
\end{equation}

Smoothly dyadically cut the sums over $p_i$ into blocks of size $p_i \asymp P_i$, inputting a smooth partition of unity $V$ such that $\sum V(p_i/P_i) = 1$ for all $i \in \{1, \ldots, \ell\}$. We then recognize the function as 
\begin{equation}
    f(\xi) = H\left(\xi, \frac{p}{p_1 \cdots p_\ell}\right) J_{k-1}(\xi),
\end{equation}
where 
\begin{equation}
    H(\xi, \lambda) := \Psi\left(\frac{X}{\xi}\sqrt{\lambda}\right) \quad \text{with} \quad X = \frac{4\pi L_1L_2\sqrt{p_1 \cdots p_\ell e^2}}{cQ}.
\end{equation}

We can then follow \textit{mutatis mutandis} the proof of \cite[Lemma 6.1]{bcl} to deduce from the smoothness and the compact support of $\Psi$ that the $\mathbb{R}^2$-Fourier transform of $H$ is rapidly decaying, viz.
\begin{equation}
    \widehat{H}(u,v) \ll_A ((1+|u|)(1+|v|))^{-A}, 
\end{equation}
for all $A \geqslant 1$, by repeatedly integrating by parts. Moreover, since $\Psi$ is compactly supported, we can represent the above by a plateau function $W(\xi/X)$. By Fourier inversion, we have
\begin{equation}
    f(\xi) = J_{k-1}(\xi) W\left( \frac{\xi}{X}\right) \iint_{\mathbb{R}^2} \widehat{H}(u,v)e(u\xi + v\tfrac{p}{P}) dudv.
\end{equation}
Inserting it in the above expression for $\Sigma_2$, we get 
\begin{align}
    \Sigma_2 &= \frac{2\pi i^{-k}}{N(Q)} \sum_{\substack{(L_1, L_2)=1 \\ L_1L_2 < L_0}} \frac{\mu(L_1L_2)}{L_1L_2} \prod_{p \mid L_1} (1-p^{-2})^{-1} \sum_{r \mid L_1} \frac{\mu(r)}{r^2} \sum_{d \mid L_2} \mu(d) \sum_{P_i, \mathrm{dyadic}} \\
    & \qquad \sum_{\substack{e \mid L_2^\infty \\ e < E}} \frac{1}{e} \iint_{\mathbb{R}^2} \widehat{H}(u, v) \sum_{c \geqslant 1} \sum_{p_i \neq p_j} \frac{1}{\sqrt{p_1 \cdots p_\ell}} \prod_{i=1}^\ell V\left(\frac{p_i}{P_i}\right) e\left( v\frac{p_i}{P_i}\right) S(u,p_1 \cdots p_\ell) dudv,
\end{align}
where we let
\begin{equation}
    S(u,p) := \sum_{n \geqslant 1} \frac{S(e^2, p, cL_1rdn)}{cL_1rdn}h_u\left( \frac{4\pi\sqrt{pe^2}}{cL_1rdn}\right)
\end{equation}
with $h_u(x) = W(x/X) J_{k-1}(x) e(u x)$. This expression $S(u,p)$  exactly appears as an arithmetic side of a Kuznetsov trace formula, with the new level $cL_1rd$. By the Kuznetsov trace formula from Proposition \ref{prop:ktf}, the innermost sums can be rephrased as
\begin{equation}
    \sum_{c \geqslant 1} \sum_{p_i} \prod_{i=1}^\ell \frac{1}{\sqrt{p_i}} e\left( c\frac{p_i}{P_i}\right) V\left( \frac{p_i}{P_i}\right) \left[\mathcal{D}(c,\mathfrak{p},u) + \mathcal{C}(c,\mathfrak{p},u) + \mathcal{H}(c, \mathfrak{p}, u)\right], 
\end{equation}
where $\mathfrak{p} = p_1 \cdots p_\ell$ and $\mathcal{D}$, $\mathcal{C}$ and $\mathcal{H}$ stand for the discrete, continuous and holomorphic contributions from the spectral side respectively. Recall that they are explicitly defined by
\begin{align}
    \mathcal{D}(c, p, u) & = \sum_j \frac{\overline{\alpha_j}(e^2) \alpha_j(p) \sqrt{pe^2}}{\cosh \pi \kappa_j} h_+(\kappa_j), \\
    \mathcal{C}(c, p, u) & = \frac{1}{\pi} \sum_{\mathfrak{c}} \int_{\mathbb{R}} \frac{\sqrt{pe^2}}{\cosh \pi t} \overline{\varphi_{\mathfrak{c}}}(e^2, t) \varphi_{\mathfrak{c}}(p, t) h_+(t) dt, \\
    \mathcal{H}(c, p, u) & = \frac{1}{2\pi} \sum_{\substack{\ell \geqslant 2 \\ 1 \leqslant j \leqslant \theta_\ell(cL_1, rd)}} (l-1)! \sqrt{pe^2}  \ \overline{\psi_{j, \ell}}(e^2) \psi_{j, \ell}(p) h_h(\ell),
\end{align}
where the precise notations are as in Section \ref{subsubsec:ktf}. 

\begin{prop}
\label{prop:bcl}
    With the above notations, we have 
    \begin{align}
        \sum_{c \geqslant 1} \sum_{p \asymp P} \frac{1}{\sqrt{p}} e\left( v\frac{p}{P}\right) V\left(\frac{p}{P}\right) \mathcal{D}(c, p, u) & \ll Q^\varepsilon (1+|u|)^2 (1+|v|)^2 \frac{\sqrt{P}}{Q}, \\
        \sum_{c \geqslant 1} \sum_{p \asymp P} \frac{1}{\sqrt{p}} e\left( v\frac{p}{P}\right) V\left(\frac{p}{P}\right) \mathcal{H}(c, p, u) & \ll Q^\varepsilon (1+|u|)^2 (1+|v|)^2 \frac{\sqrt{P}}{Q}, \\
        \sum_{c \geqslant 1} \sum_{p \asymp P} \frac{1}{\sqrt{p}} e\left( v\frac{p}{P}\right) V\left(\frac{p}{P}\right) \mathcal{C}(c, p, u) & \ll Q^\varepsilon (1+|u|)^2 (1+|v|)^2 \left(\frac{\sqrt{P}}{Q} + P^{1/4+\varepsilon} \right).
    \end{align}
\end{prop}
We obtain the analogue statement for the product by an immediate induction.
The spectral aspect of the computations are exactly as in \cite[Propositions 6.2 and 6.3]{bcl} where they are extensively treated, and we therefore afford not to give all the details. The main point to import the computations from \textit{loc. cit.} is that the quantities only differ by logarithmic factors, while the bound ultimately obtained for $\Sigma_2$ therein displays a power savings in $Q$, see \cite[End of Section 6]{bcl}. Since~$P \ll c(f) \leqslant Q$, Proposition \ref{prop:bcl} indeed implies Proposition \ref{prop:first-moments}. \qed

We briefly explain how to bound the discrete and holomorphic parts, following \cite[Proposition~6.2]{bcl}.
The discrete contribution (the holomorphic one is analogous, and easier since we have the Deligne bound) displays sums that are
\begin{equation}
    \sum_{c \geqslant 1} \sum_p \frac{1}{\sqrt{p}} e\left(v\frac{p}{P}\right) \mathcal{D}(c, p, u) V\left( \frac{p}{P} \right)= \sum_{c \geqslant 1} \sum_j \frac{e \overline{\rho_j(e^2)}}{\cosh \pi \kappa_j} h_+(\kappa_j) \sum_p \frac{\sqrt{p} \rho_j(p)}{\sqrt{p}} e\left(v\frac{p}{P}\right) V\left(\frac{p}{P}\right).
\end{equation}

We start bounding this innermost $p$-sum: 
\begin{lem}
We have 
\begin{equation}
       \sum_p \frac{\sqrt{p} \rho_j(p)}{\sqrt{p}} e\left( \frac{vp}{P} \right) V\left( \frac{p}{P}\right) \ll (|\rho_j(1)| + |\rho_f(1)|) (cL_1rd)^\varepsilon \log(P)^\varepsilon (1+|v|)^2, 
\end{equation}
where $f$ is a suitable oldform below $f$, see \cite{bcl}. 
\end{lem}

\begin{proof}
    By Mellin inversion applied to $eV$, we have 
    \begin{equation}
         \sum_p \frac{\sqrt{p} \rho_j(p)}{\sqrt{p}} e\left( \frac{vp}{P}\right) V\left( \frac{p}{P}\right) = \frac{1}{2i\pi} \sum_p \frac{\sqrt{p \alpha_j(p)}}{\sqrt{p}} V_0\left( \frac{p}{P}\right) \int_{(0)} p^{-s} \tilde{W}(s) P^s ds
    \end{equation}
    where $W(x) = W_v(x) = e(vx) V(x).$ We have that $\tilde{W}_v(it) \ll ((1+|v|)/(1+|t|))^{A}$ by integrating by parts. Swapping the summation and integration, we therefore get
    \begin{equation}
        \frac{1}{2i\pi} \int_{(0)} P^s \tilde{W}(s) \sum_p \frac{\alpha_j(p) \sqrt{p}}{p^{1/2+it}} V_0(p/P).
    \end{equation}

    By the "trivial bound" on the power one terms from Lemma \ref{lem:1-parts}, we have that 
    \begin{align}
        \sum_p \frac{\alpha_j(p) \sqrt{p}}{p^{1/2+it}} V_0\left( \frac{p}{P}\right)& \ll |\rho_j(1)| \log(cL_1 rd) \log(X)^\varepsilon + |\rho_f(1)| (cL_1 rd)^\varepsilon \\
        & \ll (|\rho_j(1)| + |\rho_f(1)|) (cL_1 rd)^\varepsilon (1+|t|)^\varepsilon \log(P)^\varepsilon.
    \end{align}

Therefore, applying the decay of $\tilde{W}(s)$ written above with $A=2$ to ensure the convergence of the vertical integral, we get that the whole sum over $p$ is indeed 
\begin{equation}
    \ll (|\rho_j(1)| + |\rho_f(1)|) (cL_1 rd)^\varepsilon (1+|v|)^2 \int_{(0)} \frac{dt}{(1+|t|)^{2-\varepsilon}}
\end{equation}
and the vertical integral converges, giving the claimed result.
    \end{proof}

Back to the whole discrete contribution, and inputting the above bound, we get that 
\begin{align}
    \mathcal{D}(c, p, u) & = \sum_{c \geqslant 1} \sum_j \frac{e \rho_j(e^2)}{\cosh \pi \kappa_j} h_+(\kappa_j) \sum_p \frac{\sqrt{p}\alpha_j(p)}{\sqrt{p}} e\left( \frac{vp}{P}\right) V\left( \frac{p}{P}\right) \\
    & \ll \sum_{c \geqslant 1} \min(X^{k-1}, X^{-1/2}) (cL_1 dr)^\varepsilon \log(P)^\varepsilon (1+|v|)^{2} e^{1+\varepsilon} \frac{1 + |\log X|}{F^{1-\varepsilon}} \\
    & \qquad \times \sum_j \frac{|\rho_f(1)|(|\rho_j(1)| + |\rho_f(1)|)}{\cosh \pi \kappa_j} \left( \frac{F}{1 + \kappa_j}\right)^C
\end{align}
where we used the bound on $h_+$  given in \cite{bcl}, with in particular $F \asymp (1+|u|)(1+4\pi L_1L_2\sqrt{Pe^2}/cQ)$, as well as the bound $e\rho_j(e^2) \ll e^{1+\varepsilon} |\rho_f(1)|$ on the coefficients. We can now sum the spectral part depending on the position of the spectral parameter $\kappa_j$ with respect to $F$: 
\begin{equation*}
    \sum_j \frac{|\rho_f(1)|(|\rho_j(1)| + |\rho_f(1)|)}{\cosh \pi \kappa_j} \left( \frac{F}{1 + \kappa_j}\right)^C \ll \sum_{|\kappa_j| < F} \frac{|\rho_j(1)|^2}{\cosh \pi \kappa_j} \frac{1}{F^{1-\varepsilon}} + \sum_{|\kappa_j| > F} \frac{|\rho_j(1)|^2}{\cosh \pi \kappa_j} \frac{1}{F^{1-\varepsilon}} \left( \frac{F}{1+|\kappa_j|}\right)^{2+\varepsilon}.
\end{equation*}
The first sum is bounded by a spectral large sieve. We can then finish the proof as in \cite{bcl}, \textit{mutatis mutandis}.
It then remains to bound the continuous spectrum:
we can do it as in \cite{bcl} following exactly the same lines, inputting the induction as in \cite{miller}, and just changing the "trivial bound" from~\cite[Lemma 2.12]{miller} to the "trivial bound" from Lemma \ref{lem:1-parts}. This completes the proof of Proposition \ref{prop:first-moments} and of Theorem~\ref{prop:moments}. 
\end{proof}

\subsection{Moment method and distribution}

As in \cite{rs}, the determination of the moments in Theorem \ref{prop:moments} essentially allows to say that $P(f,x)$, hence the central value, mimicks the behavior of a normal distribution, in phase with the Keating-Snaith Conjecture \ref{conj:keating-snaith}. 
We encapsulate in the following statement the distributional consequence of this moment method.
\begin{coro}
\label{coro:moments}
We have, for all sequence $(b_f)_{f \in H_k(q)}$, 
\begin{equation}
    \sumh_{\substack{f \in H_k(q) \\ P(f,x)/\sqrt{\log\log x} \in (\alpha, \beta)}} b_f = (M(\alpha, \beta) + o(1))\sumh_{f \in H_k(q)} b_f,
\end{equation}
where 
\begin{equation}\label{eq:M_alpha_beta}
    M(\alpha, \beta) := \frac{1}{\sqrt{2\pi}}\int_\alpha^\beta e^{-x^2/2} dx.
\end{equation}
\end{coro}

\begin{proof}
Asymptotically, Theorem \ref{prop:moments} proved that the $\ell$-th moment of $P(f,x)/\sqrt{\log\log x}$ behaves as the $\ell$-th moment of the normal distribution, i.e. when $x$ grows to infinity,
\begin{equation}
     \sum_{q \geqslant 1} \Psi\left(\frac{q}{Q}\right) \sumh_{f \in H_k(q)} \left( \frac{P(f,x)}{\sqrt{\log\log x}} \right)^\ell b_f \sim \frac{1}{\sqrt{2\pi}}\int_{\mathbb{R}} x^\ell e^{-x^2/2} dx \sum_{q \geqslant 1} \Psi\left(\frac{q}{Q}\right)\sumh_{f \in H_k(q)} b_f,
\end{equation}
for all $\ell \geqslant 0$, so we deduce that, for any polynomial $R \in \mathbb{R}[X]$, 
\begin{equation}
    \sum_{q \geqslant 1} \Psi\left(\frac{q}{Q}\right) \sumh_{f \in H_k(q)} R\left( \frac{P(f,x)}{\sqrt{\log\log  x}} \right) b_f \sim \frac{1}{\sqrt{2\pi}}\int_{\mathbb{R}} R(x) e^{-x^2/2} dx \sum_{q \geqslant 1} \Psi\left(\frac{q}{Q}\right)\sumh_{f \in H_k(q)} b_f,
\end{equation}
and by approximating the characteristic function $\mathbf{1}_{(\alpha, \beta)}$ in $L^1$-norm by a polynomial $R$, we deduce that (inputting the smooth sum over levels and the weighted sum of $f \in H_k(q)$ in the summation over $\mathcal{F}_Q$ to ease notations)
\begin{align}
    \sum_{\substack{f \in \mathcal{F}_Q \\ P(f,x)/\sqrt{\log\log x} \in (\alpha, \beta)}} b_f & 
    = \sum_{f \in \mathcal{F}_Q} \mathbf{1}_{(\alpha, \beta)}\left( \frac{P(f,x)}{\sqrt{\log\log  x}} \right) b_f \\
    & \sim \frac{1}{\sqrt{2\pi}} \int_{\mathbb{R}} \mathbf{1}_{(\alpha, \beta)}(x) e^{-x^2/2} dx \sum_{f \in \mathcal{F}_Q} b_f = M(\alpha, \beta)\sum_{f \in \mathcal{F}_Q} b_f,
\end{align}
as claimed.
\end{proof}

\subsection{Uncorrelation lemma}

A similar result has to be available when weighted by one-level densities, analogously to the central result \cite[Proposition 3, second part]{rs}: 
\begin{coro}[Weighted moments property]
\label{prop:weighted-moments}
Assume the generalized Riemann hypothesis for the symmetric squares $L$-functions $L(s, \mathrm{sym}^2 f)$.
We have, for all smooth function $h$ with compactly supported Fourier transform, and all $\ell \geqslant 1$, 
\begin{equation}
\label{weighted-moments}
\frac{1}{N(Q)} \sum_{q \geqslant 1} \Psi\left( \frac{q}{Q} \right) \sumh_{f \in H_k(q)} P(f,x)^\ell D(f,h)  = (M_\ell + o(1)) (\log\log(x))^{\ell/2} \int_{\mathbb{R}} W_{\rm O} h.
\end{equation}
\end{coro}

This proposition means that we can decouple the one-level density statement and the moment property, both exploiting trace formulas. In other words, one-level densities and sums over coefficients are \textit{uncorrelated}.

\begin{proof}
We have to study the sum
\begin{equation}
\frac{1}{N(Q)} \sum_{q \geqslant 1} \Psi\left( \frac{q}{Q} \right) \sumh_{f \in H_k(q)} P(f,x)^\ell D(f,h).
\end{equation}
The one-level density is understood by Proposition \ref{prop:explicit-formula-zeros} and can be written as 
\begin{equation}
 D(f, h) = \hat{h}(0) + \tfrac12 h(0) + P^{(1)}(f,h) + O\left( \frac{\log \log q}{\log q}\right),
\end{equation}
as proven for instance in \cite[(4.25)]{ils}, consequence of the generalized Riemann Hypothesis for $L(s, \mathrm{sym}^2 f)$, and where the implied constant only depends upon the test-function $h$.
Note that the main term of this expression is $\hat{h}(0) + \tfrac12 h(0) = \int h W_{\rm O}$, the limiting one-level density, and can therefore be pulled out of the sum, since independent of $f$ and $q$, and the Theorem~\ref{prop:moments} is therefore applicable as it stands, giving  a contribution of 
\begin{equation}
 (M_\ell + o(1)) (\log\log x )^{\ell/2} \int_{\mathbb{R}} W_{\rm O}h, 
\end{equation}
which already accounts for the main term displayed in Corollary \ref{prop:weighted-moments}.
The error term contributes negligibly to the whole sum over the family. The remaining contribution is 
\begin{equation}
\frac{1}{N(Q)} \sum_{q \geqslant 1} \Psi\left( \frac{q}{Q} \right) \sumh_{f \in H_k(q)} P(f,x)^\ell P^{(1)}(f, h),
\end{equation}
which, by applying the Cauchy-Schwarz inequality, is bounded by
\begin{equation}
    \left( \frac{1}{N(Q)} \sum_{q \geqslant 1} \Psi\left( \frac{q}{Q} \right) \sumh_{f \in H_k(q)} P(f,x)^{2\ell} \right)^{1/2} \left( \frac{1}{N(Q)} \sum_{q \geqslant 1} \Psi\left( \frac{q}{Q} \right) \sumh_{f \in H_k(q)} P^{(1)}(f, h)^2\right)^{1/2}.
\end{equation}
By Theorem \ref{prop:moments}, the first parenthesis is bounded by $(\log\log(x))^\ell$, so that its square root has similar size as the expected main term. The statement \cite[Proposition 4.1]{miller}, where they study the moments of the on-level density, bounds the second parenthesis by $O(1/\log Q) = o(1)$, proving that the whole contribution coming from $P^{(1)}(f, h)$ is negligible, as claimed.
\end{proof}

\section{Proof of Theorem \ref{thm:thm}}
\label{sec:conclusion}

With the above tools being now at hand, we follow the strategy presented in \cite{rs} in the case of quadratic twists of an elliptic curve. We show that there are not many small zeros by an amplification process, which will be used to prove that the sum over zeros in the explicit formula~\eqref{eq:explicit-central-value} contributes as an error term. The moment method will then allow to select the values for which we are in the desired range, giving the result. 

\subsection{Amplification of small zeros}

The following result, analogue of \cite[Lemma 1]{rs}, uses Theorem \ref{prop:moments} to quantify the proportion of $f \in H_k(q)$ such that $P(f,x)$ falls into a specific range; and Corollary \ref{prop:weighted-moments} to jointly quantify the proportion of $f\in H_k(q)$ having not too many small zeros. Introduce the notation $x = X^{1/\log \log \log X}$ for this section.

\begin{prop}
\label{prop:amplification}
The smooth averaged number of $f\in H_k(q)$ such that $P(f,x) / \sqrt{\log \log x} \in (\alpha, \beta)$ and such that there are no zeros with $|\gamma_f| \leqslant (\log X \log\log X)^{-1}$ is at least
\begin{equation}
\label{amplification}
   \frac{5}{8} M(\alpha, \beta) N(Q),
\end{equation}
where $M(\alpha, \beta)$ is the normal distribution, as defined in \eqref{eq:M_alpha_beta}.
\end{prop}

\begin{proof}
Choose for $h$ the explicit F\'ejer kernel up to the maximal Fourier support~$L=4$ allowed by the low-lying zero result given in Theorem \ref{thm:ilsbcl}, i.e.
\begin{equation}
    h_0(x) := \left( \frac{\sin \pi x}{\pi x}\right)^2 \qquad \hat{h}_0(y) = \max(1-|y|, 0), 
\end{equation}
which has Fourier transform supported in $(-1,1)$,
and $h(x) = h_0(4x)$ so that $\hat{h}(y) = \tfrac14\hat{h}_0(x/4)$ is compactly supported in $(-4, 4)$. Let $H = D(f,h)$ and $\Psi = \Psi(q/Q)$ to lighten notation for the duration of the proof. We get from Corollary \ref{coro:moments}:
\begin{align}
    \sum_{\substack{f \in H_k(q) \\ P(f,x)/\sqrt{\log\log x} \in (\alpha, \beta)}} H\Psi& \sim M(\alpha, \beta)\sum_{f \in H_k(q)} H\Psi,
\end{align}
and, by Corollary \ref{prop:weighted-moments}, we get
\begin{equation}
\sum_{f \in H_k(q)} H\Psi \sim \int_{\mathbb{R}} W_{\rm O}h \sum_{f \in H_k(q)} \Psi = \frac34 \sum_{f \in H_k(q)} \Psi,
\end{equation}
because
\begin{equation}
    \int W_{\rm O}(y) h(y) dy = \frac{3}{4}
\end{equation}
by the explicit choice of $h$ --- see \cite{ils} for the proof of the optimality of this function in such a setting.

We use the similar amplification argument as in \cite{rs} approach. Rewrite the above sum as 
\begin{equation}
  \sum_{(\alpha, \beta)} H \Psi = \sum_{\substack{(\alpha, \beta) \\ \exists}} H\Psi + \sum_{{\substack{(\alpha, \beta) \\ \nexists}}} H\Psi
\end{equation}
using $\ell = (\log X \log\log X)^{-1}$ and the following notation: 
\begin{align}
    \sum_{(\alpha, \beta)} H \Psi & = \sum_{\substack{f \in H_k(q) }} H\Psi \mathbf{1}_{P(f,x)/\sqrt{\log\log x} \in (\alpha, \beta)}, \\
    \sum_{\substack{(\alpha, \beta) \\ \exists}} H\Psi & = \sum_{\substack{f \in H_k(q) \\ \exists \ |\gamma_f| \leqslant \ell}} H\Psi \mathbf{1}_{P(f,x)/\sqrt{\log\log x} \in (\alpha, \beta)}, \\
    \sum_{{\substack{(\alpha, \beta) \\ \nexists}}} H\Psi & = \sum_{\substack{f \in H_k(q) \\ \nexists \ |\gamma_f| \leqslant \ell}} H\Psi \mathbf{1}_{P(f,x)/\sqrt{\log\log x} \in (\alpha, \beta)}.
\end{align}

The weights $ h(\gamma_f)$ are non-negative, since the function $h$ we chose is non-negative. If~$L(s,f)$ has a zero $\gamma_f$ of size at most $\ell$, then $\tilde\gamma_f$ is at most $\log\log(X)^{-1}$, and its conjugate is also a zero of the same size. Choosing a continuous function $h$ such that $h(0) = 1$, when $x$ grows to infinity, both~$h(\tilde{\gamma}_f)$ and $h(\tilde{\overline{\gamma}}_f)$ are at least $1-\varepsilon$, for any given $\varepsilon>0$. In particular $H = D(f, h) \geqslant 2 - \varepsilon$ and we can therefore write
\begin{equation}
      \sum_{(\alpha, \beta)} H \Psi = \sum_{\substack{(\alpha, \beta) \\ \exists}} H\Psi + \sum_{{\substack{(\alpha, \beta) \\ \nexists}}} H\Psi \geqslant (2-\varepsilon) \sum_{\substack{(\alpha, \beta) \\ \exists}} \Psi + \sum_{{\substack{(\alpha, \beta) \\ \nexists}}} H\Psi = (2-\varepsilon)\sum_{(\alpha, \beta)} \Psi + \sum_{{\substack{(\alpha, \beta) \\ \nexists}}} (H-2+\varepsilon)\Psi, 
\end{equation}
so that we get
\begin{equation}
\label{zzz}
 \sum_{(\alpha, \beta)} H\Psi - \sum_{\substack{(\alpha, \beta) \\ \nexists}} (H-2+\varepsilon)\Psi \geqslant (2-\varepsilon)\sum_{(\alpha, \beta)} \Psi.
\end{equation}
On the other hand, the above consequences of the moment method and of the limiting one-level density results allow to estimate the sums over all forms with restrictions on $P(f,x)$. More precisely, Corollary \ref{coro:moments} implies
\begin{align}
  \label{z} \sum_{(\alpha, \beta)} \Psi &\sim M(\alpha, \beta) \sum_{f \in H_k(q)} \Psi
\end{align}
and Corollary \ref{prop:weighted-moments} states that
\begin{align}
    \label{zz} \sum_{(\alpha, \beta)} H \Psi &\sim \frac{3}{4} M(\alpha, \beta) \sum_{f \in H_k(q)} \Psi.
\end{align}
We thus derive from \eqref{zzz} that
\begin{equation}
 \frac34 M(\alpha, \beta)  \sum_{f \in H_k(q)} \Psi - \sum_{\substack{(\alpha, \beta) \\ \nexists}} (H-2+\varepsilon)\Psi \geqslant (2-\varepsilon) M(\alpha, \beta) \sum_{f \in H_k(q)} \Psi.
\end{equation}
Since $0 \leqslant h \leqslant 1$, we get
\begin{equation}
    (2-\varepsilon)\sum_{\substack{(\alpha, \beta) \\ \nexists}} \Psi \geqslant \sum_{\substack{(\alpha, \beta) \\ \nexists}} (2-H-\varepsilon)\Psi \geqslant \frac54 M(\alpha, \beta) + o(1), 
\end{equation}
from where we obtain a lower bound for the smoothed quantity of $f \in H_k(q)$ having zeros of size at most $\ell$, viz.
\begin{equation}
    \sum_{\substack{(\alpha, \beta) \\ \nexists}} \Psi\left( \frac{q}{Q} \right) \geqslant (\tfrac58 - \varepsilon) M(\alpha, \beta) \sum_{f \in H_k(q)} \Psi\left( \frac{q}{Q} \right), 
\end{equation}
for all $\varepsilon>0$, as desired. 
\end{proof}

\begin{rmk}
The constant $5/8$ is exactly the one appearing in Theorem \ref{thm:thm}, and this is where we see that the quality of the results towards the density conjecture, i.e. the width of the allowed Fourier support, conditions the quality of this lower bound. Note that this gives the same value as the method in \cite{ils} to obtain lower bounds for nonvanishing, as anticipated by \cite{rs}.
\end{rmk}

\subsection{Few zeros contributing a lot}

The following result quantifies how rare are the $f \in H_k(q)$ such that the contribution from the sum over zeros in the explicit formula \eqref{eq:explicit-central-value} is large.

\begin{prop}
\label{prop:few-zeros}
For $x \leqslant q \leqslant Q$, the number of $f \in H_k(q)$ such that 
\begin{equation}
\label{few-zeros}
    \sum_{|\gamma_f| \geqslant (\log X \log\log X)^{-1}} \log(1+(\gamma_f \log x)^{-2}) \geqslant (\log\log\log(X))^3
\end{equation}
is asymptotically dominated by $ X/\log \log \log X$.
\end{prop}

\begin{proof}
The same proof as in \cite[Lemma 2]{rs} holds \textit{mutatis mutandis}.
\end{proof}

\subsection{Conclusion}

This closely follows the argument of \cite{rs}, now that all the corresponding estimates have been established. We write it here for the sake of completeness. Recall from Proposition \ref{prop:explicit-formula-central-values}, with $x = c(f)$, that
\begin{equation}
    \log L(\tfrac12,f) = P(f,x) - \tfrac12 \log\log(x) + O\left(\sum_{\gamma_f} \log(1 + (\gamma_f \log x)^{-2})\right).
\end{equation}

By Proposition \ref{prop:amplification}, we may select $f$'s such that $P(f,x)/\sqrt{\log\log X} \in (\alpha, \beta)$ and that there are no small zeros, without loosing at most a proposition of $\tfrac38$ of the whole family, i.e.
\begin{equation}
 \sum_{\substack{f \in H_k(q) \\ P(f,x)/\sqrt{\log\log x} \in (\alpha, \beta) \\ \nexists |\gamma_f| \leqslant (\log X \log\log X)^{-1}}} 1 \geqslant \tfrac58 M(\alpha, \beta) N(Q).
\end{equation}

By Proposition \ref{prop:few-zeros}, we may remove $f$'s such that the sum over zeros larger than $(\log X \log \log X)^{-1}$ contributes more than $(\log\log\log(X))^3$, since they are asymptotically a negligible cardinality, and the other ones do not contribute that much.

The proportion of $f$ such that $P(c(f),f)/\sqrt{\log \log c(f)}$ falls into $(\alpha, \beta)$ is therefore asymptotically larger than $\tfrac58 M(\alpha, \beta)$ as claimed in the theorem, henceforth ending the proof of Theorem \ref{thm:thm}.

\subsection*{Acknowledgements} We are grateful to Maksym Radziwi\l{}\l{} and Kannan Soundararajan for enlightening discussions. We also thank Steven J. Miller and Pico Gilman for further comments.
This work started when D. L. was visiting Kyushu University and ended when A. I. S. was visiting Université de Lille; we thank both institutions for good working environment. We acknowledge support from the Labex CEMPI (ANR-11-LABX-0007-01), the~CNRS (PEPS), and JSPS KAKENHI Grant Number 22K13895.

\end{document}